\documentclass[11pt]{article}
\usepackage{amssymb, authblk}
\usepackage{amsmath, bbm}
\usepackage{fullpage}
\usepackage{amsthm, natbib, hyperref}
\usepackage{graphicx}
\usepackage{verbatim}

\usepackage{algorithm}% http://ctan.org/pkg/algorithms
\usepackage{algpseudocode}% http://ctan.org/pkg/algorithmicx
\usepackage[dvipsnames]{xcolor}
 \usepackage{tikz}
\usetikzlibrary{arrows.meta}
\usetikzlibrary{decorations.pathreplacing}

\algnewcommand\algorithmicinput{\textbf{INPUT:}}
\algnewcommand\INPUT{\item[\algorithmicinput]}
\algnewcommand\algorithmicoutput{\textbf{OUTPUT:}}
\algnewcommand\OUTPUT{\item[\algorithmicoutput]}

\usepackage{hyperref}[]
\hypersetup{
    colorlinks=true,
    linkcolor=blue,
    filecolor=magenta,      
    urlcolor=cyan,
      citecolor=blue,
    }

\usepackage{cleveref}

\newtheorem{theorem}{Theorem}
\newtheorem{lemma}[theorem]{Lemma}

\newtheorem{definition}{Definition}
\newtheorem{remark}{Remark}
\newtheorem{assumption}{Assumption}

\newcommand{\tf}{\widetilde f^{s,e}}
\newcommand{\p}{\mathcal P}

\newcommand{\kse}{\kappa^{s,e}_{\max}}

\allowdisplaybreaks

\DeclareMathOperator*{\argmin}{arg\,min}
\DeclareMathOperator*{\argmax}{arg\,max}

\date{\vspace{-5ex}}

\title{Univariate Mean Change Point Detection: \\Penalization, CUSUM and Optimality}
\author[1]{Daren Wang}
\author[2]{Yi Yu}
\author[3]{Alessandro Rinaldo}

\affil[1]{\small Department of Statistics, University of Chicago}
\affil[2]{\small School of Mathematics, University of Bristol}
\affil[3]{\small Department of Statistics and Data Science, Carnegie Mellon University}

\begin{document}
\maketitle
\begin{abstract}
The problem of univariate mean change point detection and localization based on a sequence of $n$ independent observations with piecewise constant means has been intensively studied for more than half century, and serves as a blueprint for change point problems in more complex settings. We provide a complete characterization of this classical problem in a general framework in which the upper bound $\sigma^2$ on the noise variance, the minimal spacing $\Delta$ between two consecutive change points and the minimal magnitude $\kappa$ of the changes, are allowed to vary with $n$. We first show that consistent localization of the change points, when the signal-to-noise ratio $\frac{\kappa \sqrt{\Delta}}{\sigma} < \sqrt{\log(n)}$, is impossible. In contrast, when $\frac{\kappa \sqrt{\Delta}}{\sigma}$ diverges with $n$ at the rate of at least $\sqrt{\log(n)}$,  we demonstrate that two computationally-efficient  change point estimators, one based on the solution to  an $\ell_0$-penalized least squares problem and the other on the popular wild binary segmentation algorithm, are both consistent and achieve a localization rate of the order $\frac{\sigma^2}{\kappa^2} \log(n)$. We further show that such rate is minimax optimal,  up to a $\log(n)$ term.

%One estimator is an $\ell_0$ penalized least squares estimator of the change points, while the 
%In this paper, we set up the problem in a general framework, regarding the upper bound of fluctuations $\sigma$, minimal spacing $\Delta$ and minimal jump size $\kappa$, allowing all three parameters to change as the sample size $n$ grows unbounded.  Within this general framework, we will first provide two information-theoretic results regarding the lower bounds of the signal-to-noise ratio and the localization rate.  We will then investigate two widely used types of methods: the penalization-based method and the CUSUM-based method.  We will show that there are computationally-efficient algorithms from each of these two categories enjoying matching upper bounds.  

\end{abstract}

\section{Introduction}
\label{section:intro}
%!TEX root = ./draft.tex

Research on change point detection in time series data has a relatively long history in modern statistics, covering both online \citep[e.g.][]{Wald1945, Page1954, JamesEtal1987} and offline \citep[e.g.][]{vostrikova1981detection, YaoAu1989} search problems.  It has been recently going through a renaissance due to the routinely collected complex and large amount of data sets in the `Big Data' era.  Change point detection problems in high-dimensional means \citep[e.g.][]{ChoFryzlewicz2015, Cho2015, AstonKirch2014, Jirak2015, wang2016high}, in covariance structures \citep[e.g.][]{AueEtal2009, avanesov2016change, WangEtal2017}, in dynamic networks \citep[e.g.][]{GibberdRoy2017, WangEtal2018}, and in sequentially-correlated time series \citep[e.g.][]{Lavielle1999, DavisEtal2006, AueEtal2009} have been actively studied in recent years.  

Arguably, the simplest and best-studied change point detection problem is on univariate mean from independent observations.  It is fair to say that this is the most important ingredient in more complex problems.  We formalize the model in \Cref{assume:change}.

\begin{assumption}[Model]\label{assume:change}

Let $Y_1, \ldots, Y_n \in \mathbb{R}$ be independent sub-Gaussian random variables with continuous density such that $\mathbb{E}(Y_i) = f_i$ and $\|Y_i\|_{\psi_2} \leq \sigma$ for all $i \in \{1, \ldots, n\}$.  

Let $\{\eta_k\}_{k=0}^{K+1} \subset \{0, \ldots, n\}$ be a collection of change points such that $0 = \eta_0 < \eta_1 < \ldots < \eta_K <  \eta_{K+1} = n$ and 
	\[
		f_{\eta_{k} +1} = \ldots =  f_{\eta_{k+1}}, \quad \mbox{for all } k = 0, \ldots, K.
	\]

Assume the minimal spacing $\Delta$ and the jump size $\kappa$ satisfy
	\[
		\min_{k = 1, \ldots, K+1} \bigl\{\eta_k - \eta_{k-1}\bigr\} \geq \Delta > 0,
	\]	
	and
	\[
		\min_{k = 1, \ldots, K+1} \bigl|f_k - f_{k-1}\bigr| = \min_{k = 1, \ldots, K+1} \kappa_k = \kappa > 0.
	\]

\end{assumption} 

\begin{remark}
In fact we do not need the condition that $Y_i$'s have continuous densities.  We include it here for simplicity, such that we do not need to consider the event in which that two sets of random variables have the same sample mean.  This is the only place this condition is used.
\end{remark}

The model is completely characterized by the sample size $n$, the upper bound $\sigma$ on the fluctuations in terms of Orlicz-$\psi_2$-norm\footnote{For any random variable $X$, let $\|X\|_{\psi_2}$ be its Orlicz-$\psi_2$ norm, i.e.
	\[
	\|X\|_{\psi_2} = \sup_{k \geq 1} \bigl\{\mathbb{E}\bigl(|X|^k\bigr)\bigr\}^{1/k}.
	\]}, the minimal spacing $\Delta$ between two consecutive change points and the lower bound $\kappa$ of the jump size in terms of the absolute value of the difference between two consecutive population means.  All three parameters $\sigma$, $\Delta$ and $\kappa$ are allowed to change as $n$ grows.  Since the number of change points $K$ is upper bounded by $n/\Delta$, we will not keep track of $K$, as its upper bound can be derived from the other parameters.

The goal of a change point detection problem is to obtain {\bf consistent} change point estimators $\{\hat{\eta}_k\}_{k = 1}^{\widehat{K}}$, with $\widehat{\eta}_1 < \widehat{\eta}_2 < \ldots < \widehat{\eta}_{\widehat{K}}$, such that
	\begin{equation}\label{eq:eta}
	\widehat{K} = K \quad \mbox{and} \quad \max_{k = 1, \ldots, \widehat{K}} \bigl|\hat{\eta}_k - \eta_k\bigr| %= \max_{k = 1, \ldots, \widehat{K}}\epsilon_k 
	\leq \epsilon(n) = \epsilon,
	\end{equation}
	where $\epsilon/n \to 0$, with probability tending to 1 as $n \to \infty$.  In the rest of the paper, we will refer to the sequence $\epsilon/n$ as  the {\bf localization rate}.  Notice that the inequality in \eqref{eq:eta} can be seen as providing an upper  bound on the Hausdorff distance between $\{ \eta_k\}_{k=1}^K$ and $\{ \widehat{\eta}_k\}_{k-1}^{\widehat{K}}$, both viewed as subsets of $\{1,\ldots,n\}$; see \eqref{hausdorf} below.
	
In order to quantify the difficulty of the problem, we rely on the quantity 
	\begin{equation}\label{eq-def-SNR}
		\kappa\sqrt{\Delta}/\sigma, 
	\end{equation}
	which can be thought of  as measuring the {\bf signal-to-noise ratio}.  As we will see, the intrinsic statistical hardness of the change point detection and localization problems is fully captured by this quantity. In particular, the difficulty of the problem increases as $\kappa$ and $\Delta$ decrease, and $\sigma$ increases.  \Cref{eq-def-SNR} is rooted in two-sample mean testing (with known variance), resembling $t$-statistics used therein, and has counterparts in high-dimensional mean, covariance and network change point detection problems \citep[e.g.][]{wang2016high, WangEtal2017, WangEtal2018}. 
	
With the previously defined localization rate and signal-to-noise ratio, the optimality of the estimators possesses two aspects.
	\begin{itemize}
		\item [(i)]	Consistency. The first natural question one might ask is under what conditions localization is itself possible. We tackle this problem by identifying combinations of the model parameters, which we express using the signal-to-noise ratio  \eqref{eq-def-SNR}, for which no estimator of the change points is guaranteed to be consistent, in a minimax sense. 
		\item [(ii)] Outside the region of impossibility identified in the previous step, the second natural question is to derive a lower bound on the localization rate that holds for any estimator. Once the information-theoretic lower bound is established, one may then proceed to demonstrate a computationally-efficient algorithm whose localization rate matches such lower  bound. This algorithm is therefore minimax optimal.   %One would like to derive the lower bound of $\epsilon/n$, despite how large the signal-to-noise ratio can be. 
	\end{itemize}
	
We would like to point out that the phase transition phenomenon in terms of signal-to-noise ratio for the localization that we demonstrate below in Section~\ref{section:lb} has been shown previously found in the literature.  For instance, Theorem~1 in \cite{chan2013} showed a phase transition for testing the presence of a single change point that matches the one we obtain for localization.   \cite{FrickEtal2014} have further generalized this type of detection results to allow for an unbounded number of change points.  In term of localization, Theorem~2.8 in \cite{FrickEtal2014} has also provided a localization error rate that match the minimax rate we derive in this paper.  Similar results can also be found in other papers including \cite{dumbgen2001multiscale}, \cite{dumbgen2008multiscale}, \cite{LiEtal2017}, \cite{jeng2012simultaneous}, \cite{enikeeva2018bump}, to name but a few.

%Once the information-theoretic lower bounds are established, it is of vital interest to provide computationally-efficient algorithms with upper bounds matching the information-theoretic lower bounds.  One can roughly categorize change point detection methods into two families: the penalization-based method and the CUSUM-based method.
In this article we will be focusing on two types of change point estimators, one based on penalized least squares and the other on CUSUM statistics. Both types of estimator have been thoroughly studied. 

\begin{itemize}
	\item  
	 %Since the introduction of fused Lasso \citep{TibshiraniEtal2005}, there have been a number of papers using $\ell_1$-type penalties to recover piecewise-constant (or polynomial) signals, including e.g. \cite{Rinaldo2009}, \cite{HarchaouiLevy2010}, \cite{Tibshirani2014} and \cite{KevinNIPS}.  Being a convex relaxation of $\ell_0$-penalty, the aforementioned $\ell_1$-based change point detection methods, when used in locating change points, sacrifice accuracy in trade of computational efficiency. 
		There exist several results and algorithms for change point detection using $\ell_0$ penalization, including  \cite{LiebscherWinkler1999}, \cite{FriedrichEtal2008}, \cite{BoysenEtal2009} and \cite{KillickEtal2012}.   It is worth comparing three papers providing theoretical results based on $\ell_0$-penalization methods.  \cite{LavielleMoulines2000} studied the $\ell_0$-penalization approach under general distributions, and showed that if one chooses the penalization parameter $\lambda$ properly, then one would get similar asymptotic results to the case where the model assumes Gaussian noise.  The closest-related result there is Theorem~9, which only showed asymptotic results.  In this paper, we obtain the lower bounds based on Gaussian noise, but the upper bounds are achieved for sub-Gaussian noise, and provide non-asymptotic results.   \cite{BoysenEtal2009} studied consistent estimation of a general class of functions  based on the solution of the $\ell_0$ least squares problem given in \Cref{eq-uhat} below, which they referred to as the Potts functional. In particular, under the assumption that the mean function is piecewise-constant with a fixed number of change points, the authors show that a solution to \eqref{eq-uhat} can consistently locate the change points if the minimal spacing satisfies $\Delta=cn$ for some $0<c<1$ and the change size $\kappa$ is a constant.  We extend such results by allowing all the parameters in the model -- namely $\kappa$, $\Delta$ and $\sigma$ -- to change with $n$ at a nearly minimax rate, and will demonstrate the existence of a phase transition in the space of model parameters. Furthermore, our analysis is non-asymptotic. \cite{FanGuan2017} studied the $\ell_0$-denoising on a general class of graphs including chains, i.e. piecewise-constant time series signals, and provided a number of information-theoretic results.  Our paper and theirs have different targets -- we focus on the change point localization but theirs focus on prediction, which are complementary to each other.  
		
		There are also a number of papers in 1980's studying the univariate mean change point detection problem from the least squares estimators perspective, for instance, \cite{YaoDavis1986}, \cite{Yao1988}, \cite{YaoAu1989}.  The change point estimators are derived from least squares estimators, and the number of change points are chosen via the Schwarz' information criterion.  It can be shown \citep[e.g.][]{TickleEtal2018} that the Schwarz' information criterion is asymptotically equivalent to the $\ell_0$ penalization.  Note that the results obtained there are asymptotic, while ours are non-asymptotic and allow all parameters to vary as the sample size $n$.  Another related area is the reduced isotonic regression problem, which assumes the monotonic signal is piecewise-constant and which aims to recover the signal.  \cite{GaoEtal2017} has shown an iterated logarithmic lower bound when there are multiple change points.  Despite the close connection, the focus and results thereof are different from ours.
		
		It is worth mentioning that $\ell_0$-penalization method is appealing from the computational aspect, at least in the univariate case.  \cite{FriedrichEtal2008} showed that \eqref{eq-uhat} can be computed using dynamic programming and its computational cost is of order $O(n^2)$.  \cite{KillickEtal2012} introduced the pruned exact linear time (PELT) method, which has the worst case computational cost of order $O(n^2)$; while in the situations where the number of change points increases linearly with $n$, the expected time of PELT is of order $O(n)$.  There are also other algorithms, including \cite{Rigaill2010} and \cite{MaidstoneEtal2017}, which have been shown to have an expected cost which is smaller than that of PELT, but which have the worst case cost also of order $O(n^2)$.

	\item The CUSUM (see \Cref{def-cusum}) is short for the cumulative sums, proposed in \cite{Page1954} for an online change point problem, and has been a cornerstone in numerous change point detection methods.  We will show in \Cref{section:bs} that in the univariate situation, it is identical to the likelihood ratio test statistics to test whether or not there exists a change point.  Binary segmentation (BS) \citep[e.g.][]{ScottKnott1974, vostrikova1981detection} based on CUSUM statistics has been shown to be consistent, yet optimal, in locating the change points.  In the last few years, a considerable amount of efforts have been made into developing variants of BS in order to handle multiple change points scenarios, see e.g. \cite{fryzlewicz2014wild}, \cite{BaranowskiEtal2016} and \cite{EichingerKirch2018}.
	
	An important reference is \cite{fryzlewicz2014wild}, who put forward the wild binary segmentation (WBS) algorithm,  is a variant of BS, and provide an analysis of its performance. Unfortunately, the proof Theorem 3.2 in that reference suffers from critical errors. In this paper we rectify those issues and present a more comprehensive analysis of WBS that keeps track explicitly of all the relevant parameters and, in particular, allows to conclude  that the localization rate afforded by WBS is nearly minimax rate optimal. Although our efforts in this regard are non-trivial, we acknowledge that the results we derive in \Cref{section:bs} and the proofs in \Cref{sec:app-3} borrow heavily from \cite{fryzlewicz2014wild}.  As a result, we provide optimal results with all parameters being allowed to change with $n$ and weaker conditions.
\end{itemize}

The univariate mean change point detection problem has been studied intensively, and we are aware that the results in this paper have been produced in different forms in existing literature.  However, we still see the need to produce this paper merely focusing on this simple scenario, providing systematical analysis on various theoretical points, which can be served as benchmarks in more modern challenges.  

We summarize our contributions as follows.
	\begin{itemize}
		\item [(i)] We describe a phase transition in space of the model parameter that separates parameter combinations for which consistent change point estimation is impossible (in a minimax sense) from those for which there exist algorithms that are provably consistent. Furthermore, we provide a global information-theoretic lower bound on the localization rate that holds over most of the region of the parameter space for which consistent estimation is possible. It is worth pointing out that this same phrase transition could be indirectly deduced from the existing literature on minimax change point detection and on change point localization for univariate piecewise signals. See, in particular, \cite{chan2013} and \cite{FrickEtal2014}.  Here we provide a direct proof of this phenomenon based on formal minimax arguments.

		\item [(ii)] We demonstrate that the $\ell_0$-penalization method produces a minimax rate-optimal estimator of the change points.  %To the best of our knowledge, this result is novel the optimality of the $\ell_0$-penalization method is derived here for first time.   
		In addition, we demonstrate that the localization error rate of $\ell_0$-penalization method is locally adaptive to the  jump size at each change point, a desirable feature both in theory and in practice (see \Cref{remark:local adaptive}). 

		\item [(iii)] Among CUSUM-based methods, we show that the WBS algorithm put forward by \cite{fryzlewicz2014wild} is also minimax rate-optimal.  While our analysis of the WBS is heavily inspired by the proof techniques in \cite{fryzlewicz2014wild}, we are able to provide more refined results  with optimal tracking of the underlying parameters, thus obtaning optimak rates. We also require weaker conditions than in \cite{fryzlewicz2014wild}.  
	\end{itemize}

The  paper is organized as follows.  The information-theoretic results are exhibited in \Cref{section:lb}.  Matching upper bounds provided by $\ell_0$-penalization method and WBS can be found in Sections~\ref{section:potts} and \ref{section:bs}, respectively.  Most of the proofs and technical details are in the Appendices.

%\subsection{Notation}\label{sec-notation}

%For any vector $v$, let $\|v\|_0$ de its $\ell_0$-norm.  For any two vectors of the same dimension $v$ and $w$, let $(v, w) = v^{\top}w$ be the inner product.   For any set $S$, let $|S|$ be its cardinality.  

\section{Phase Transition and Optimality Minimax Rates}
\label{section:lb}
% !TEX root = ./draft.tex

Recall the two aspects of optimality we describe in \Cref{section:intro}: to identify parameter combinations for which  consistent localization is possible and to determine a minimax lower bound on the localization rate.  In \Cref{lemma-low-snr} we describe the low signal-to-noise ratio regime for which estimating the location of the change points cannot be done. In detail,  we show that if
	\begin{equation}\label{eq-low-snr}
		\kappa\sqrt{\Delta}/\sigma < \sqrt{\log(n)},
	\end{equation}
	then no consistent estimator of the locations of the change points exists. On the other hand, when $\kappa\sqrt{\Delta}/\sigma \geq \sqrt{\log(n)}$, \Cref{lemma-error-opt} demonstrates a minimax lower bound on the localization rate of the form $\frac{\sigma^2}{\kappa^2 n}$, for all $n$ large enough. 
The analysis of the localization  procedures described in Sections~\ref{section:potts} and \ref{section:bs} will confirm that these results are in fact quite sharp. Specifically, we will verify both the existence of a phase transition for the localization task as the  signal-to-noise ratio crosses the threshold $\sqrt{\log(n)}$, as prescribed by \Cref{lemma-low-snr},  and the near minimax optimality of the lower bound of  \Cref{lemma-error-opt}.

Below, for two subsets $E_1$ and $E_2$ of $\{1,\ldots,n\}$, we let 
\begin{equation}\label{hausdorf}
H(E_1,E_2 ) = \max \Big\{  \max_{x \in E_1} \min_{y \in E_2} | x-y|, \max_{y \in E_1} \min_{x \in E_2} | x-y|\Big\}
\end{equation}
denote their Hausdorff distance. 
		
\begin{lemma}\label{lemma-low-snr}
Let $\{Y_i\}_{i=1}^n$ be a time series satisfying \Cref{assume:change} and let $P^n_{\kappa, \Delta, \sigma}$ denote the corresponding joint distribution.  For any $0 < c < 1$, consider the class of distributions
	\[
		\mathcal{P}^n_{c} = \left\{P^n_{\kappa, \Delta, \sigma}: \, \Delta = \min \left\{\left\lfloor c\frac{\log(n)}{\kappa^2/\sigma^2}\right\rfloor, \left\lfloor \frac{n}{4} \right\rfloor\right\} \right\}.
	\]
	Then, there exists an $n(c)$, which depends on $c$, such that, for all $n$ larger than $n(c)$, 
	\[
		\inf_{\hat{\eta}} \sup_{P \in \mathcal{P}^n_c} \mathbb{E}_P\bigl( H(\hat{\eta}, \eta(P) ) \bigr) \geq \frac{n}{8},
	\]
	where the infimum is over all estimators $\widehat{\eta} = \{ \widehat{\eta}_k\}_{k=1}^{\widehat{K}}$ of the change point locations and $\eta(P)$ is the set of locations of the change points of $P \in  \mathcal{P}^n_c$.
\end{lemma}

In the above result, it is possible to let $c \to 0$ as $n \to \infty$ (and in fact, the value of $n(c)$ is increasing in $c$). Thus, we conclude that, if $\kappa\sqrt{\Delta}/\sigma < \lfloor \sqrt{\log(n)} \rfloor < \lfloor n/4 \rfloor $, the localization rate is bounded away from $0$, i.e. the estimator is not consistent.  %In fact, $c$ can also be a function of $n$ so that $c \to 0$ as $n \to \infty$. This would correspond to  even smaller signal-to-noise ratios and would similarly lead to the same conclusion.  %Therefore, \Cref{lemma-low-snr} completely characterizes the parameter space corresponding to \eqref{eq-low-snr}.  %The proof of \Cref{lemma-low-snr}, which can be found in the Appendix, utilizes Fano's Lemma and the techniques introduced in \cite{GaoEtal2017}.

In our next result we  complement \Cref{lemma-low-snr} by showing that if instead 
\begin{equation}\label{eq-hi-snr}
		\kappa \sqrt{\Delta}/\sigma \geq \zeta_n,
	\end{equation}
for any sequence $ \{ \zeta_n \}_{n=1,2,\ldots}$ of positive numbers diverging to infinity at an arbitrary pace  as $n \rightarrow \infty$, then the corresponding lower bound is at least of order $\frac{\sigma^2}{\kappa^2}$, for all $n$ large enough.  Of course, in light of \Cref{lemma-low-snr}, this lower bound is interesting only when $\zeta_n$ is larger than $\sqrt{\log (n)}$. In the next sections, we will further show that, provided that $\zeta_n$ is of the order $\sqrt{\log^{1 + \xi}(n)}$ or larger, for any $\xi>0$, then, up to a logarithmic factor in $n$, $\frac{\sigma^2}{\kappa^2}$ yields the asymptotic minimax lower bound on the localization rate.

\begin{lemma}\label{lemma-error-opt}
Let $\{Y_i\}_{i=1}^n$ be a time series satisfying \Cref{assume:change} with one and only one change point.	 Let $P^n_{\kappa, \Delta, \sigma}$ denote the corresponding joint distribution.  Consider the class of distributions
	\[
		\mathcal{Q}^n = \left\{P^n_{\kappa, \Delta, \sigma}: \, \Delta < n/2, \, \kappa\sqrt{\Delta}/\sigma \geq \zeta_n \right\},
	\]
	for any sequence $\{ \zeta_n \}$ such that $\lim_{n \rightarrow \infty} \zeta_n = \infty $. %Let $\hat{\eta}$ and $\eta(P)$ be the estimator and the true change point, respectively.  
	Then, for all $n$ large enough, it holds that 
	\[
		\inf_{\hat{\eta}} \sup_{P \in \mathcal{Q}^n} \mathbb{E}_P\bigl(\bigl|\hat{\eta} - \eta(P)\bigr|\bigr) \geq \max \left\{ 1, \frac{1}{2} \Big\lceil\frac{\sigma^2}{\kappa^2} \Big\rceil e^{-2} \right\}, %\max\left\{c\sigma^2/\kappa^2, \, 1/2\right\},
	\]
	where the infimum is over all estimators $\widehat{\eta}$ of the change point location and $\eta(P)$ denotes the change point location of $P \in \mathcal{Q}^n$.
	
\end{lemma}

The bounds in \Cref{lemma-low-snr} and \Cref{lemma-error-opt} are slightly sharper than the minimax  lower bounds obtained by taking $p=1$ in  Proposition 3 in the supplementary material of \cite{wang2016high}. Indeed, our analysis allows for a more refined characterization of the phase transition for the localization task by exhibiting  the threshold value of $\sqrt{\log n}$ describing the transition from the low to high signal-to-noise ratio regime.

\section{$\ell_0$ Penalization}
\label{section:potts}
%!TEX root = ./draft.tex

In this section we describe an estimator of the change point locations based on the $\ell_0$ penalty and demonstrate  that it is minimax rate optimal.

We first formalize the $\ell_0$-penalized optimization problem, and define the change point estimators generated therefrom.  For the sake of analysis, we will provide an alternative objective function, which, we will show, generates identical change point estimators.

For fixed tuning parameter $\lambda > 0$ and data $\{Y_i\}_{i=1}^n$, define the $\ell_0$-penalized sum of squares objective function as 
	\begin{equation}\label{eq-H}
	H(u, \{Y_i\}_{i=1}^n, \lambda) = \sum_{i=1}^n \bigl(Y_i - u_i\bigr)^2 + \lambda \|Du\|_0,
	\end{equation}
	where $\|\cdot\|_0$ is the $\ell_0$-norm of a vector, $D \in \{\pm 1, 0\}^{(n-1) \times n}$ satisfies $(Du )_j = u_{j+1}-u_j$, for $j \in \{1, \ldots, n-1\}$.  Let 
	\begin{equation}\label{eq-uhat}
	\widehat{u}(\lambda) = \argmin_{u \in \mathbb{R}^n} H(u, \{Y_i\}_{i=1}^n, \lambda).
	\end{equation}
	Let $\bigl\{\hat{\eta}_k\bigr\}_{k = 1}^{\widehat{K}}$ be the collection 
	\[
		\mathcal{J}\bigl(\widehat{u}\bigr) = \bigl\{i \in \{1, \ldots, n-1\}:\, \widehat{u}_i\bigl(\lambda\bigr) \neq \widehat{u}_{i+1}\bigl(\lambda\bigr) \bigr\}.
	\]
	We thus call $\bigl\{\hat{\eta}_k\bigr\}_{k = 1}^{\widehat{K}}$ the change point estimator induced by $\widehat{u}\bigl(\lambda\bigr)$, or the change point estimator from the optimization problem \eqref{eq-uhat}.  If one replace the penalty term $\|Du\|_0$ with the $\ell_1$-norm $\|Du\|_1$, then \eqref{eq-H} is the fused Lasso objective function, see e.g. \cite{TibshiraniEtal2005} and \cite{Rinaldo2009}.

Alternatively, let $\mathcal{P}$ be any {\it interval partition} of $\{0, 1, \ldots, n\}$, i.e. a collection of $\mathcal{P}_k$ disjoint subsets of $\{1,\ldots,n\}$ of the form
	\[
	\mathcal{P} = \bigl\{\{1, \ldots, i_1\}, \{i_1 + 1, \ldots, i_2\}, \ldots, \{i_{\mathcal{P}_{k}-1} + 1, \ldots, i_{\mathcal{P}_k} \}\bigr\},
	\]
	for some integers $0 < i_1 < \cdots < i_{\mathcal{P}_k} = n$, where $\mathcal{P}_{k} \geq 1$. In particular, if $\mathcal{P}_k = 1$, then $\mathcal{P} = \bigl\{\{1, \ldots, n\}\bigr\}$.
For a fixed positive tuning parameter $\lambda > 0$ and data $\{Y_i\}_{i=1}^n$, let 
\begin{equation}\label{eq-p-hat}
	\widehat{\mathcal{P}}(\lambda) =\argmin_{\mathcal P} G\bigl(\mathcal P, \{Y_i\}_{i=1}^n, \lambda\bigr).
	\end{equation}
where the minimum ranges over all interval partitions of $\{1,\ldots,n\}$ and, for any such partition $\mathcal{P}$,
	\begin{equation}\label{eq-G}
	G\bigl(\mathcal P, \{Y_i\}_{i=1}^n, \lambda\bigr) = \sum_{I\in \mathcal P} \sum_{i\in I}\bigl(Y_i - \overline{Y}_I\bigr)^2 + \lambda \bigl(| \mathcal P| - 1\bigr), 
	\end{equation}
	with $\overline{Y}_I = |I|^{-1} \sum_{i\in I} Y_i$.
	The optimization problem \eqref{eq-p-hat} is known as the minimal partition problem and can be solved using dynamic programming in polynomial time \citep[e.g. Algorithm~1 in][]{FriedrichEtal2008}.   
The change point estimator resulting from the solution to \eqref{eq-p-hat} is simply obtained from taking  all the right endpoints of the intervals  $I \in \widehat{\mathcal{P}}$ except $n$.  In general,  without assuming any conditions on the inputs, there is no guarantee that the minimizers are unique.

We now make the simple observation that the optimization problems \eqref{eq-uhat} and \eqref{eq-p-hat} with the same inputs yield the same change point estimators. To see this equivalence we will introduce some notation that we will be using throughout.  For any vector $v \in \mathbb{R}^n$ and any $i \in \{1, \ldots, n-1\}$, if $v_i \neq v_{i+1}$, one calls $i$ an induced change point of $v$, and the collection of all the induced change points of $v$ is denoted as $J(v)$.  The set $J(v)$ yields an interval partition, i.e., if $J(v) = \{i_1, \ldots, i_N\}$, then one can define the interval partition induced by $v$ as 
	\[
	\mathcal{P} = \bigl\{\{1, \ldots, i_1\}, \{i_1 + 1, \ldots, i_2\}, \ldots, \{i_{N} + 1, \ldots, n\}\bigr\}.
	\]
	Conversely, for any interval partition $\mathcal{P}$ and a sequence $\{Y_i\}_{i=1}^n$, define their induced piecewise-constant vector $v$ as $v_i = \overline{Y}_I$, for any $i \in I$ and $I \in \mathcal{P}$.  
		Since for $I \subset \{1, \ldots, n\}$,
		\[
		\overline{Y}_I = \argmin_{x \in \mathbb{R}} \sum_{i \in I}(Y_i - x)^2,
		\]
		it follows that with the same inputs $\{Y_i\}_{i=1}^n$ and $\lambda > 0$, the solutions to \eqref{eq-uhat} and \eqref{eq-p-hat} induce each other in the sense specified above.	

\begin{remark}[Tuning parameter]
	If we view any vector $u \in \mathbb{R}^n$ as a step function with at most $n-1$ jumps, then the tuning parameter $\lambda$ penalizes the number of jumps in $u$.  For an integer interval $I \subset \{1, \ldots, n\}$, the tuning parameter $\lambda$ works in the following way.  If an integer interval $I$ is to be split into two integer sub-intervals $I_1$ and $I_2$, then it follows from \Cref{lemma-pre} that the sum of squares will decreases by 
		\begin{equation}\label{eq-remark-lambda}
			\frac{|I_1| |I_2|}{|I_1| +|I_2|}  \bigl(\overline Y_{I_1} - \overline Y_{I_2}\bigr)^2,
		\end{equation}
		but, at the same time, the penalty term will increase by $\lambda$.  Therefore the trade-off guiding the choice between refining a candidate integral partition of $\{1,\ldots,n\}$ by introducing one additional split and leaving it unchanged (so that this partition must then provide an optimal solution to \eqref{eq-p-hat}), is described by comparing \eqref{eq-remark-lambda} to $\lambda$.
		In \Cref{prop:1d localization} we will provide a theoretically optimal choice for $\lambda$. %, but in practice, one would expect to use data-driven methods, e.g. cross validation, to choose this tuning parameter.
\end{remark}

\begin{remark}
In the rest of this paper, when there is no ambiguity, we allow the following abuse of notation.  If $s < e$, $s, e \in \mathbb{Z}$, we sometimes refer $\{s, s+1, \ldots, e\}$ and $\{s+1, \ldots, e\}$ as $[s, e]$ and $(s, e]$, respectively.  
\end{remark}

\subsection{Optimal change point localization}
\label{section:change point}

Recall in \Cref{lemma-low-snr} we have shown that if $\kappa\sqrt{\Delta}/\sigma < \sqrt{\log(n)}$, then no algorithm is guaranteed to produce consistent change point estimators.  To demonstrate the performances of \eqref{eq-uhat}, we thus require the signal-to-noise ratio  $\kappa\sqrt{\Delta}/\sigma$ to be larger than a diverging function of $n$, which we take to be of the form  $\log^{(1+\xi)/2}(n)$. As remarked in the previous section, such choice is  consistent with \Cref{lemma-error-opt}, which in principle allows for a vanishing localization rate. 

\begin{assumption}\label{assum-phase}
	There exists a sufficiently large absolute constant $C_{\mathrm{SNR}} > 0$ such that for any $\xi > 0$,
	\[
		\kappa\sqrt{\Delta}/\sigma \geq C_{\mathrm{SNR}}\sqrt{\log^{1+\xi}(n)}.
	\]
\end{assumption}
We remark that the introduction of the parameter $\xi > 0$ is to guarantee that even if $\Delta \asymp n$, the resulting estimator is still consistent.   
We do not know whether the above assumption can be relaxed by allowing for a rate of increase for $\kappa\sqrt{\Delta}/\sigma$ slower than   $\sqrt{\log^{1+\xi}(n)}$. In our proofs, this seems to be the slowest rate that we can afford.

\begin{theorem}\label{prop:1d localization}
	Let $\{Y_i\}_{i=1}^n$ satisfy \Cref{assume:change} and, for any $\lambda>0$, set
	\begin{equation}\label{eq:minimizer}
		\widehat  u(\lambda) = \argmin_{u\in \mathbb R^n} H(u, \{Y_i\}_{i=1}^n, \lambda).
		\end{equation}
	%where $H(\cdot, \cdot, \cdot)$ is defined in \eqref{eq-H} and $\lambda = C_{\lambda}\sigma^2 \log(n)$ for a sufficiently large absolute constant $C_{\lambda} > 0$.	 
	Let $\{\widehat v_{k} (\lambda) \}_{k=1}^{\widehat{K}(\lambda)}$ be the collection of change points induced by $\widehat u(\lambda)$.
	Under \Cref{assum-phase}, for any choice of $c>3$, there exists a constant  $C_{\lambda} > 0$, which depends on $c$ such that, for $\lambda = C_{\lambda}\sigma^2 \log(n)$, it holds that
	\[
		\mathbb{P} \bigl\{\widehat{K}(\lambda) = K,\, \mbox{ and } |\widehat \nu_k(\lambda) -\nu_k| = \epsilon_k \le C_{\epsilon}\sigma^2 \log(n) /\kappa^2_k, \, \forall k \in \{1,\ldots,K\} \bigr\} \geq 1 - e\cdot n^{3-c},
	\]
	where $C_{\epsilon} > 0$ is a constant depending on $C_{\lambda}$ and $C_{\mathrm{SNR}}$.  
\end{theorem}
Recalling \Cref{lemma-error-opt}, we see now that the error bound we derived in \Cref{prop:1d localization} is minimax rate optimal aside from possibly  a $\log(n)$ factor.
\Cref{prop:1d localization} shows that with a proper choice of the tuning parameter, \eqref{eq-uhat} provides consistent change point estimators in the sense that with probability tending to 1 as $n \to \infty$, it holds that $\widehat{K}(\lambda) = K$ and for all $k \in \{1, \ldots, K\}$,
\[
\epsilon_k/n \leq C_{\epsilon}\frac{\sigma^2}{\kappa^2}\frac{\log(n)}{n} \leq C_{\epsilon}\frac{\Delta}{n \log^{\xi}(n)} \to 0,
\]
as $n \to \infty$.  

\begin{remark}[Uniqueness]
	We mentioned earlier that the  minimizers of the optimization problems \eqref{eq-H} and \eqref{eq-p-hat} need not be unique.  However, if the independent errors have a continuous distribution, as assumed in \Cref{assume:change}, the minimizer is unique almost surely, for each $n$ and each $\lambda$; if not, then any two solutions, say $\mathcal{P}$ and $\mathcal{P}'$, are such that 
		\[
			\sum_{I\in \mathcal P} \sum_{i\in I}\bigl(Y_i - \overline{Y}_I\bigr)^2 -  \sum_{I'\in \mathcal P'} \sum_{i\in I'}\bigl(Y_i - \overline{Y}_{I'}\bigr)^2 = \lambda \bigl( | \mathcal P'| - \mathcal P|  \bigr). 
		\] 
		This is a quadratic polynomial in the $\{Y_i\}_{i=1}^n$.  The set of its real solutions (if any exists) has $n$-dimensional Lebesgue measure $0$.  In general, if there are multiple solutions (so that \Cref{assume:change} does not hold), \Cref{assum-phase} guarantees that, with large probability (more precisely, in the event $\mathcal{B}$ defined in the proof), the minimizer is unique almost surely.
%	 In \Cref{assume:change}, we impose the condition that the data $\{Y_i\}_{i=1}^n$ are generated from a distribution possessing a continuous density function.  This condition and \Cref{assum-phase} guarantee that, with large probability (more precisely, in the event $\mathcal{B}$ defined in the proof), the minimizer is unique. 
\end{remark}

\begin{remark}\label{remark:local adaptive}
	It is natural to expect that change point localization should be locally adaptive in a sense that, if the jump size $\kappa_k$ gets larger, then it is easier to estimate the location of the change point $\eta_k$.  In fact, the error rate $\epsilon_k$ derived in \Cref{prop:1d localization} matches this feature.
\end{remark}

\begin{proof}
Define the event 
	\[
	\mathcal{B} = \biggl\{\sup_{1\le a < b< c\le n} \sqrt{\frac{(b-a)(c-b)}{c-a}} \bigl|\overline Y_{(a+1,b]} - \overline f_{(a+1,b]} + \overline Y_{(b+1,c]} - \overline f_{(b+1,c]} \bigr| \leq \sigma \sqrt{ C_{ \lambda} \log(n)} \biggr\},
	\]
	where $C_{\lambda} > 0$ is a large enough constant only depending on $c$, and $a, b, c$ are integers.  In the remainder of the proof we work on the event $\mathcal{B}$. By \Cref{lem:max} in \Cref{sec:app-1}, this occurs with probability at least $1 - e\cdot n^{3-c}$.
	
For simplicity we will remove the dependence on $\lambda$ in our notation as it will implicitly understood that $\lambda = C_{\lambda}  \sigma^2 \log (n)$. Let $\widehat{\mathcal{P}}$	be the interval partition induced by $\widehat{u}$ (see  \eqref{eq:minimizer}), and let $\{s+1, \ldots, e\}$ be any member of $\widehat{\mathcal{P}}$.  The proof is completed by showing the following four steps.
	\begin{itemize}
	\item [Step 1] The interval $(s, e]$ contains no more than two true change points.  This is shown in \Cref{lemma:three changes}.%, with the input $\lambda$ satisfying
%		\[
%			2 \sigma^2 \log(n) \le\lambda \le  \kappa_k^2 \Delta /32.
%		\]

	\item [Step 2] If $(s, e] $ contains exactly two true change points, say $\eta_k,\eta_{k+1}$, then
		\[
		\eta_k - s \leq C_{\epsilon}\sigma^2 \log(n)/\kappa^2_k, \mbox{ and } e -\eta_{k+1} \leq C_{\epsilon}\sigma^2 \log(n)/\kappa^2_{k+1}.
		\]
		This is shown in \Cref{lemma:two change localization}.% by choosing $\lambda = C_{\lambda}\sigma^2\log(n)$, $C_{\lambda} > 2$ and $C'' = 6C_{\lambda}$.
		
	\item [Step 3] If $(s, e]$ contains only one true change point, say $\eta_k$, without loss of generality, let $\eta_k - s \leq e - \eta_k$, then it must hold that
		\[
		s \leq \eta_k \leq e \leq \eta_{k+1} 
		\]
		and
		\[
		\eta_k - s \leq C_{\epsilon}\sigma^2 \log(n)/\kappa^2_k, \mbox{ and } \eta_{k+1} - e \leq C_{\epsilon}\sigma^2 \log(n)/\kappa^2_{k+1}.
		\]
		This is shown in \Cref{section:one change}.%, by choosing a large enough $C_{\lambda}$, and $C'' = 8C_{\lambda}$.
		
	\item [Step 4] If $(s, e]$ contains no true change point, then there exist two true change points $\eta_k$ and $\eta_{k+1}$ satisfying
		\[
		\eta_k \leq s < e \le \eta_{k+1}
		\]
		and
		\[
		s - \eta_k \leq C_{\epsilon}\sigma^2 \log(n)/\kappa^2_k, \mbox{ and } \eta_{k+1} - e \leq C_{\epsilon}\sigma^2 \log(n)/\kappa^2_{k+1}.
		\]
		This is shown in \Cref{section:no change}.%, by choosing $C_{\lambda} > 2$ and $C'' = 8C_{\lambda}$.
	\end{itemize}
\end{proof}

\section{CUSUM}
\label{section:bs}
% !TEX root = ./draft.tex

As for the univariate mean change point detection problem, the $\ell_0$-penalization estimator is not the only one which achieves the minimax optimality.  Binary segmentation (BS) \citep[e.g.][]{ScottKnott1974} based on CUSUM statistics is arguably the most popular change point detection method.  It has been shown that BS is consistent yet not optimal \citep[e.g.][]{venkatraman1992consistency}.  \cite{fryzlewicz2014wild} proposed a variant of BS, namely wild binary segmentation (WBS), which is shown to lead to a better localization rate than the BS algorithm. In this section, we will recall the WBS algorithm, and give refined results on its performance, with a proof which has more careful tracking of all parameters and all the constants involved.  As a result, we prove that WBS, just like the method studied in the previous section, also guarantees a localization error rate that is rate minimax optimal.    However, compared to the $\ell_0$-penalization methods, WBS is computationally more expensive and involves more tuning parameters.

\begin{definition}[CUSUM statistics]\label{def-cusum}
	For a sequence $\{Y_i\}_{i=1}^n$, any pair of time points $(s, e) \subset \{0, \ldots, n\}$ with $s < e -1$, and any time point $t = s+1, \ldots, e-1$, let the CUSUM statistics be
	\[
	\widetilde Y^{s,e}_t = \sqrt {\frac{e-t}{(e-s)(t-s)} } \sum_{i=s+1}^t Y_i  -\sqrt {\frac{t-s}{(e-s)(e-t)} } \sum_{i=t+1}^e Y_i.
	\]
\end{definition}

For a collection of independent Gaussian random variables $\{Y_i\}_{i=1}^n$ with $\mathbb{E}(Y_i) = f_i$ and same variance, one can easily derive that 
	\[
		\max_{t = 1, \ldots, n-1}\bigl|\widetilde{Y}^{0, n}_t\bigr|
	\] 
	is the generalized likelihood ratio statistic to test the hypothesis:
	\begin{align}\label{eq-test}
	H_0: f_1 = \cdots = f_n \mbox{ v.s. }  H_1: \mbox{there exists } t_* \mbox{ such that } f_1 = \cdots = f_{t_*} \neq f_{t_* + 1} = \cdots f_n.
	\end{align}
In particular, the BS algorithm searches for the time point which has the largest absolute CUSUM statistics value, i.e.
	\[
		\widehat{t} = \argmax_{t = 1, \ldots, n-1}\bigl|\widetilde{Y}^{0, n}_t\bigr|.
	\] 
	%\textcolor{red}{The next sentence is not true: the NP lemma is for simple vs simple, while this hypothesis problem has a composite null and alternative: It follows from Neyman--Pearson lemma that BS based on CUSUM statistics is optimal to test \eqref{eq-test}. } 
	
However, as noted in \cite{fryzlewicz2014wild}, when there are potentially multiple change points, their combined effect might cancel out and the BS is guaranteed to be effective only when applied to intervals containing at most one change point. WBS improves on BS by performing multiple CUSUM tests over randomly chosen sub-intervals in such a manner that each change point will, with high probability, be the only change point deep inside some selected interval and can be identified using the BS algorithm within that interval. See \Cref{algorithm:WBS} for a formal description of WBS.
% and then conduct BS in those intervals.  Much effort has been made down in this stream.  We illustrate the optimality of CUSUM-based methods using WBS as an example, and the detailed algorithm thereof is collected in \Cref{algorithm:WBS}.  The key idea thereof is to utilize randomization to enter the universe where the intervals of interest containing at most one change point.  

\begin{algorithm}[htbp]
\begin{algorithmic}
	\INPUT Independent samples $\{X_i\}_{i=1}^{n}$, collection of intervals $\{ (\alpha_m,\beta_m)\}_{m=1}^M$, tuning parameters $\tau > 0$.
	\For{$m = 1, \ldots, M$}  
		\State $(s_m, e_m) \leftarrow [s,e]\cap [\alpha_m,\beta_m]$
		\If{$e_m - s_m > 1$}
			\State $b_{m} \leftarrow \argmax_{s_m + 1 \leq t \leq e_m - 1}  | \widetilde Y^{s_m,e_m}_{t}|$
			\State $a_m \leftarrow \bigl| \widetilde Y^{s_m,e_m}_{b_{m}}\bigr|$
		\Else 
			\State $a_m \leftarrow -1$	
		\EndIf
	\EndFor
	\State $m^* \leftarrow \argmax_{m = 1, \ldots, M} a_{m}$
	\If{$a_{m^*} > \tau$}
		\State add $b_{m^*}$ to the set of estimated change points
		\State WBS$((s, b_{m*}),\{ (\alpha_m,\beta_m)\}_{m=1}^M, \tau)$
		\State WBS$((b_{m*}+1,e),\{ (\alpha_m,\beta_m)\}_{m=1}^M,\tau ) $

	\EndIf  
	\OUTPUT The set of estimated change points.
\caption{Wild Binary Segmentation. WBS$((s, e),$ $\{ (\alpha_m,\beta_m)\}_{m=1}^M, \tau $)}
\label{algorithm:WBS}
\end{algorithmic}
\end{algorithm} 

It has been shown under a set of slightly stronger conditions, 
\cite{fryzlewicz2014wild} originally put forward the WBS algorithm and provided an analysis of its performance. Below we refine such analysis and formally prove that that WBS is  minimax rate-optimal in terms of the required signal-to-noise ratio and the localization rate. 

\begin{theorem}\label{thm-wbs}
For WBS algorithm detailed in \Cref{algorithm:WBS}, assume the inputs are as follows:
	\begin{itemize}
		\item the sequence $\{Y_i\}_{i=1}^n$ satisfies Assumptions~\ref{assume:change} and \ref{assum-phase};
		\item the collection of intervals $\{(\alpha_m, \beta_m)\}_{m=1}^M \subset \{1, \ldots, n\}$, whose endpoints are drawn independently and uniformly from $\{1, \ldots, n\}$, satisfy $\max_{m = 1, \ldots, M}(\beta_m - \alpha_m) \leq C_R \Delta$, almost surely, for an absolute constant $C_{R} > 1$;
		\item the tuning parameters  $\tau$ satisfies 
			\begin{equation}\label{eq-tau}
				c_{\tau, 1}\sigma\sqrt{\log(n)} < \tau < c_{\tau, 2}\kappa\sqrt{\Delta},
			\end{equation}
			where $c_{\tau, 1}, c_{\tau, 2} > 0$ are sufficiently large and   small  absolute constants.
	\end{itemize}	
Let $\bigl\{\hat{\eta}_k\bigr\}_{k=1}^{\widehat{K}}$ be the corresponding output of the WBS algorithm. Then, 
% and, for $k = 1, \ldots, K$, let 
%	\[
%		\epsilon_k = \bigl|\hat{\eta}_k - \eta_k\bigr|.
%	\]
	\begin{align}
	& \mathbb{P}\left\{\widehat{K} = K \quad \text{and} \quad   \epsilon_k \leq C_{\epsilon}\sigma^2\log(n)\kappa^{-2}_k, \forall k \in \{1,\ldots,K\} \right\} \nonumber \\
	& \hspace{2cm} \geq 1 - e\cdot n^{3 -c} - e\cdot n^{2 -c} - \exp\left\{\log\left(\frac{n}{\Delta}\right) - \frac{M\Delta^2}{16n^2}\right\}, \label{eq-wbs-thm-prob}
	\end{align}
	where $c > 3$ is an absolute constant and $C_{\epsilon} > 0$ is a sufficiently large constant.
\end{theorem}

\begin{remark}
For simplicity, we require \Cref{assume:change} in \Cref{thm-wbs}, but in fact we do not need continuous density functions condition.  In addition, we can set $\xi = 0$ in \Cref{assum-phase} if $\Delta = o(n)$.  	
\end{remark}

\Cref{thm-wbs} shows that with suitable choice for the tuning parameters, WBS is optimal in the sense that:
	\begin{itemize}
		\item under the  signal-to-noise ratio regime detailed in \Cref{assum-phase}, it yields consistent estimators of the change point locations that with probability tending to 1:  $\widehat{K} = K$, and for all $k = 1, \ldots, K$,
			\[
				\epsilon_k/n \leq C_{\epsilon}\sigma^2\log(n)\kappa^{-2}_k /n \leq \frac{C_{\epsilon}}{C^2_{\mathrm{SNR}}}\frac{\Delta}{n\log^{\xi}(n)} \to 0,
			\]
			as $n \to \infty$; and
		
		\item it possesses a localization rate 
			\[
				\epsilon/n = n^{-1}\max_{k = 1, \ldots, K} C_{\epsilon}\sigma^2\log(n)\kappa^{-2}_k \leq n^{-1}C_{\epsilon}\sigma^2\log(n)\kappa^{-2},
			\]
			which is minimax rate optimal, save for a $\log(n)$ factor, according to \Cref{lemma-error-opt}.
	\end{itemize}

\begin{remark}
To guarantee that \eqref{eq-wbs-thm-prob} tends to 1 as $n \to \infty$, the number of random intervals $M$ needs to satisfy 
	\[
		M \gtrsim \frac{n^2}{\Delta^2}\log\left(\frac{n}{\Delta}\right).
	\]	
\end{remark}

\begin{remark}[Tracking constants]
For readability, we refrain the pursuit of explicitly expressing all constants, and only show the hierarchy of the constants.  One would first choose $c > 3$ in \eqref{eq-wbs-thm-prob} to make sure that the consistency result holds.  This will determine $c_{\tau, 1}$, which is the same as $C_{\gamma}$ used in the proof, and consequently $c_{\tau, 2}$, which also depends on $C_{\mathrm{SNR}}$ and $C_R$.  All these constants finally determine $C_{\epsilon}$.  
\end{remark}

\begin{remark}[Tuning parameters]
The tuning parameter $\delta$ is introduced to avoid false positives.  Specifically, the range displayed in \Cref{eq-tau} is used in {\bf Step 1} in the proof.  Notice that, by \Cref{assum-phase} and with properly chosen constants, such range is not an empty set for $\tau$.  As shown in the proof,  over an event of probability tending to $1$, the lower bound of \eqref{eq-tau} serves as an upper bound of the maximum CUSUM statistics when there are no change points, and the upper bound serves as a lower bound of the maximum CUSUM statistics when there exists a change point. 
\end{remark}

\begin{remark}[Comparisons with \Cref{prop:1d localization}]  In \Cref{prop:1d localization}, the only tuning parameter is the penalization level $\lambda$, while in \Cref{thm-wbs}, we require knowledge on $\tau$, $\delta$, and the number of random intervals $M$.  In practice, \cite{fryzlewicz2014wild} suggest and AIC-based method for picking these parameters.  %they can all be chosen via data-driven methods, e.g. cross-validation, but bring in more computational burden.
	
\end{remark}

\begin{proof}
Since $\epsilon$ is the desired order of localization error rate, by induction, it suffices to consider any generic interval $(s, e) \subset (0, T)$ that satisfies 
	\[
		\eta_{k-1} \leq s \leq \eta_k \leq \ldots \leq \eta_{k+q} \leq e \leq \eta_{k+q+1}, \quad q\ge -1, 
	\]
	and
	\[
		\max \bigl\{ \min \bigl\{\eta_k - s,\, s - \eta_{k-1}\bigr\},\,  \min \bigl\{\eta_{k+q+1} - e,\, e - \eta_{k+q}\bigr\}\bigr\} \leq \epsilon,
	\]
	where $q = -1$ indicates that there is no change point contained in $(s, e)$.

Under \Cref{assum-phase}, it holds that
	\[
		\epsilon = C_{\epsilon}\sigma^2\log(n)\kappa^{-2} \leq \frac{C_{\epsilon}}{C^2_{\mathrm{SNR}}}\frac{\Delta}{\log^{\xi}(n)} \leq \Delta/4,
	\]	
	with sufficiently large $C_{\mathrm{SNR}}$.  It, therefore, has to be the case that for any change point $\eta_k \in (0, T)$, either $|\eta_k - s|\leq \epsilon$ or $|\eta_k - s| \geq \Delta - \epsilon \geq 3\Delta /4$.  This means that $ \min\{|\eta_k-e|,\, |\eta_k-s|\} \leq \epsilon$ indicates that $\eta_k$ is a detected change point in the previous induction step, even if $\eta_k \in (s, e)$.  We refer to $\eta_k \in (s, e)$ an undetected change point if $\min\{\eta_k - s, \, \eta_k - e\} \geq 3\Delta/4$.
	
In order to complete the induction step, it suffices to show that WBS (i) will not detect any new change point in $(s, e)$ if all the change points in that interval have been previous detected, and (ii) will find a point $b \in (s, e)$ -- in fact in $(s + \delta(e - s), e - \delta(e-s))$ -- such that $|\eta_k - b| \leq \epsilon$ if there exists at least one undetected change point in $(s, e)$.

We will consider the events $\mathcal{A}_1(\gamma)$, $\mathcal{A}_2(\gamma)$ and $\mathcal{M}$ defined in \eqref{eq-event-A1}, \eqref{eq-event-A2} and \eqref{eq-event-M}, respectively.  Set $\gamma$ to be $C_{\gamma}\sigma \sqrt{\log(n)}$, with a sufficiently large $C_{\gamma}$.  The rest of the proof assumes the the event 	
	\[
		\mathcal{A}_1(C_{\gamma}\sigma \sqrt{\log(n)}) \cap \mathcal{A}_2(C_{\gamma}\sigma \sqrt{\log(n)}) \cap \mathcal{M}.
	\]
	which, in light of \Cref{lem:A1A2M} from \Cref{sec:app-3}, has probability tending to $1$.
\vskip 3mm
\noindent {\bf Step 1.}  In this step, we will show that WBS will consistently detect or reject the existence of undetected change points within $(s, e)$.

Let $a_m$, $b_m$ and $m^*$ be defined as in \Cref{algorithm:WBS}.  Suppose there exists a change point $\eta_k \in (s, e)$ such that $ \min\{\eta_k - s, e - \eta_k\} \geq 3\Delta/4$.  In the event $\mathcal M$, there exists an interval $(\alpha_m, \beta_m)$ selected by WBS such that $\alpha_m \in [\eta_k - 3\Delta/4, \eta_k - \Delta/2]$ and $\beta_m \in [\eta_p+\Delta/2, \eta_p +3\Delta/4]$.

Following \Cref{algorithm:WBS}, $[s_m, e_m] = [\alpha_m,\beta_m]\cap [s, e]$.  We have that $\min \{\eta_k - s_m, e_m - \eta_k\} \ge (1/4)\Delta$ and $[s_m, e_m] $ contains at most one true change point.

By choosing $c_1 = 1/2$ in \Cref{lemma-step1-1}, it holds that
	\[
		\max_{s_m < t < e_m} \bigl|\widetilde{f}_{t}^{s_m, e_m}\bigr|  \geq  \kappa_k \sqrt{\Delta} /4 ,
	\]
	where $e_m-s_m \le 2\Delta$ is used in the last inequality.  Therefore
		\begin{align*}
			 a_m &  = \max_{s_m < t < e_m} \bigl|\widetilde{Y}_{t}^{s_m, e_m}\bigr| \geq \max_{s_m < t < e_m} \bigl|\widetilde{f}_{t}^{s_m, e_m}\bigr|  - \gamma  \geq  \kappa_k \sqrt{\Delta}/4 - \gamma.
		\end{align*}
	Thus for any undetected change point $\eta_k \in (s, e)$, it holds that
		\begin{equation}\label{eq:wbsrp size of population}
			a_{m^*} = \sup_{1\le m\le  M} a_m \geq  \kappa_k \sqrt{\Delta}/4  - \gamma \geq c_{\tau, 2} \kappa_k \sqrt \Delta,   
		\end{equation}
		where the last inequality is from the choice of $\gamma$ and $c_{\tau, 2} > 0$ is achievable with a sufficiently large $C_{\mathrm{SNR}}$ in \Cref{assume:change}.  Then, WBS correctly accepts the existence of undetected change points.

Suppose there does not exist any undetected change point within $(s, e)$, then for any $(s_m, e_m) = (\alpha_m, \beta_m) \cap (s, e)$, one of the following situations must hold.
	\begin{itemize}
		\item [(a)]	There is no change point within $ (s_m, e_m)$;
		\item [(b)] there exists only one change point $\eta_k \in (s_m, e_m)$ and either  $\min\{\eta_k - s_m, e_m - \eta_k\} \le \epsilon_k$; or
		\item [(c)] there exist two change points $\eta_k, \eta_{k+1} \in (s_m, e_m)$ and $  \eta_k - s_m \le \epsilon_k$, $e_m - \eta_{k+1}  \leq \epsilon_{k+1}$.
	\end{itemize}
Since cases (a) and (b) are similar and in fact simpler to the case (c), we will only provide the analysis for (c) in the proof.  
Observe that if (c) holds,  by
		\Cref{lemma:cusum two detected change point bound} 
		$$ \sup_{s_m\le t\le e_m}  | \widetilde f^{s_m,e_m}_t|   \le  \sqrt {  e_m-\eta_{k+1}}   \kappa_{ k+1}   +   \sqrt {\eta_k-s_m } \kappa_k   \le 2 C_\epsilon \sigma \sqrt {\log (n) }  .  $$
	As a result,
		$$  \sup_{s_m\le t\le e_m} | \widetilde Y^{s_m,e_m}_t |   \le 
		 \sup_{s_m\le t\le e_m}   |  \widetilde f^{s_m,e_m}_t  -\widetilde Y^{s_m,e_m}_t  |   +  \sup_{s_m\le t\le e_m} |  \widetilde f^{s_m,e_m}_t |   \le
		 2 C_\epsilon \sigma \sqrt {\log (n) }  + C_\gamma \sigma \sqrt {\log(n) } <\tau
		 $$
		 where  event $\mathcal A_1 ( C_\gamma \sigma \sqrt {\log(n) } )$  is used in the first inequality and  \eqref{eq-tau} is used in the last inequality. Therefore WBS will always correctly reject the existence of undetected change points.

\vskip 3mm
\noindent {\bf Step 2.}  Assume that there exists a change point $\eta_k \in (s, e)$ such that $\min\{\eta_k - s, \eta_k - e\} \ge 3\Delta/4$.  Let $s_m$, $e_m$ and $m^*$ be defined as in \Cref{algorithm:WBS}.  To complete the proof it suffices to show that, there exists a change point $\eta_k \in (s_{m*}, e_{m*})$ such that $\min\{\eta_k - s_{m*}, \eta_k - e_{m*}\} \geq \Delta/4$ and $|b_{m*} - \eta_k| \leq \epsilon$.

To that end, we are to ensure that the assumptions of \Cref{lem-cov-12} are verified.  The proof of \Cref{lem-cov-12} relies on a number of results, the relationship of which is shown in \Cref{fig-rm}.  Observe that \eqref{eq:wbs noise} is straightforward from \Cref{assum-phase}, \eqref{eq-wbs-lambda} and \eqref{eq:wbs noise 2} follow from the definition of $\mathcal A_1$ and $\mathcal A_2$, and that \eqref{eq:wbs size of sample} follows from \eqref{eq:wbsrp size of population}.

Thus, all the conditions in \Cref{lem-cov-12} are met, and we therefore conclude that there exists a change point $\eta_{k}$, satisfying
	\begin{equation}
		\min \{e_{m^*}-\eta_k,\eta_k-s_{m^*}\}    >  \Delta /4 \label{eq:coro wbsrp 1d re1}
	\end{equation}
	and
	\[
		| b_{m*}-\eta_{k}|\le  C_3\gamma^2 \kappa^{-2} \le \epsilon,
	\]
	where the last inequality holds from the choice of $\gamma$ and \Cref{assum-phase}.

The proof is complete by noticing the fact that \Cref{eq:coro wbsrp 1d re1}  and  $(s_{m^*}, e_{m^*}) \subset (s, e)$ imply that
	\[
	\min \{e-\eta_k,\eta_k-s\}  >  \Delta /4 >\epsilon.
	\]
	As discussed in the argument before {\bf Step 1}, this implies that $\eta_k $ must be an undetected change point.

\end{proof}

%\section{Discussion}
%label{section:dis}
%\input{Discussion}

\section{Conclusions}
In this paper we have provided a complete characterization of the classical problem of univariate mean change point localization for a sequence of independent sub-Gaussian random variables with piecewise-constant means. We have considered the most general setting in which all the parameters of the problems are allowed to change with the length $n$ of the sequence. We have identified a critical function of the model parameters that is able to discriminate the portion of the parameter space in which consistent localization is impossible from the part in which it is feasible. We have further derived the minimax optimal localization rate for this problem and showed that two computationally efficient methods achieve such a rate.

We would like to point out that the $\ell_0$-penalization methods can also be used in handling change point detection for more complex data types, such as high-dimensional mean, covariance and networks.  The developments rely on feasible algorithms for their corresponding problems, but we conjecture that $\ell_0$-penalization methods on complex data types would also enjoy the same optimality with fewer tuning parameters than those in CUSUM-based methods.

Finally, we conjecture that the upper bounds on the localization rate exhibited in both Sections~\ref{section:potts} and \ref{section:bs} can be sharpened by replacing the $\log(n)$ term with a smaller quantity of order $\log \log(n)$, thus further reducing the gap with the lower bound in \Cref{lemma-error-opt}.

\section{Acknowledgments}
We thank Zhou Fan, Paul Fearnhead, Rebecca Killick, Housen Lin, Axel Munk, Richard J. Samworth, Ryan Tibshirani and Tengyao Wang for constructive conversations.

\appendix
\section{Proofs of the Results in \Cref{section:lb}}
\label{sec:app-2}
%!TEX root = ./draft.tex

\begin{proof}[Proof of \Cref{lemma-low-snr}]
Without loss of generality, suppose that $n/4$ is an integer.  For $l \in \{1, \ldots, n/4\}$, let $\widetilde u_l \in \mathbb R^{n}$ be such that the $i$th coordinate of $\widetilde u_l (i)$, $i = 1, \ldots, n$, satisfies
	\[
		\widetilde u_l (i) = \begin{cases}
			\sqrt{c \sigma^2 \log(n)}, &  i = l; \\
			0, & \mbox{otherwise},
		\end{cases}
	\]
	where $0 < c < 1$.  Let $\widetilde v_l \in \mathbb R^{n}$ be such that $\widetilde v_l(i) = \widetilde u_l(n-i+1)$, $i = 1, \ldots, n$.  Let $\widetilde P_l$ and $\widetilde Q_l$ be the multivariate Gaussian distributions $\mathcal{N}(\widetilde u_l, \sigma^2 I_n)$ and $\mathcal{N}(\widetilde v_l, \sigma^2 I_n)$, respectively and set
 	\[
		\widetilde P = \frac{1}{n/4} \sum_{l = 1}^{n/4} \widetilde P_l \quad \mbox{and} \quad \widetilde Q = \frac{1}{n/4} \sum_{l = 1}^{n/4} \widetilde Q_l.
	\]

Note that for each $l \in \{1, \ldots, n/4\}$, $\widetilde{P}_l$ has two change points,  at locations $l - 1$ and $l$, and therefore, $\Delta = 1$.  Furthermore, the jump size is $\kappa = \sqrt{c \sigma^2 \log(n)}$ and the fluctuation is $\sigma^2$.  As a result,  
	\[
		\kappa \sqrt{\Delta}/\sigma = \sqrt{c \log(n)},
	\]
	which implies that all $\widetilde{P}_l \in \mathcal{P}^n_c$.  The same arguments show that  $\widetilde{Q}_l \in \mathcal{P}^n_c$, for all $l$.  For each $l$ and $l'$ in $\{1,\ldots,n/4  \}$, we have that, by constructions, $H( \eta(\widetilde P_l), \eta(\widetilde Q_{l'}) \geq \frac{n}{2} $, where $\eta(\widetilde P_l)$ and $\eta(\widetilde Q_{l'})$ denote the sets of change point locations of $\widetilde P_l$ and $\widetilde Q_{l'}$, respectively. Then it follows from 
 Le Cam's lemma \citep[e.g.][]{yu1997assouad} that
	\begin{equation} \label{eq:lecam}
		\inf_{\hat{\eta}} \sup_{P \in \mathcal{P}^n_c} \mathbb{E}_P\bigl( H(\hat{\eta}, \eta(P) ) \bigr) \geq \frac{n}{4}\bigl\{1 - d_{\mathrm{TV}}(\widetilde P, \widetilde Q)\bigr\},
	\end{equation}	
	where $d_{\mathrm{TV}}(\cdot, \cdot)$ is the total variation distance between two probability measures and the infimum is over all estimators $\widehat{\eta} = \{ \widehat{\eta}_k\}_{k=1}^{\widehat{K}}$ of the change point locations. Above, $\eta(P)$ is the set of locations of all the change points of $P \in  \mathcal{P}^n_c$.

Let $u_l \in \mathbb R^{n/2}$ be a sub-vector of $\widetilde{u}_l$ consisting of the first $n/2$ entries of $\widetilde u_l$.  Let 
	$P_l$ and $P_0$ be the multivariate Gaussian distributions $\mathcal{N}(u_l, \sigma^2 I_{n/2})$ and $\mathcal{N}(0, \sigma^2 I_{n/2})$, receptively.  Due to the symmetry between $\widetilde{u}_l$ and $\widetilde{v}_l$, it holds that
	\begin{equation}\label{eq:deduction}
		d_{\mathrm{TV}}(\widetilde P, \widetilde Q) \leq 2 d_{\mathrm{TV}}(P, P_0),
	\end{equation}
	where $P = \frac{1}{n/4} \sum_{l=1}^ {n/4} P_l$.  Since $d_{\mathrm{TV}}(P, P_0) \leq \sqrt {\chi^2 (P, P_0)}$, where $\chi^2(\cdot, \cdot)$ is the $\chi^2$-divergence between two probability measures \citep[see, e.g., Equation 2.27 in][]{Tsybakov2009}, it suffices to provide an upper bound for $\chi^2(P, P_0)$.
We have
	\begin{align*}
		\chi^2 (P, P_0) & = \left(\frac{1}{n/4}\right)^2 \sum_{l, m = 1}^{n/4}  \mathbb{E}_{P_0} \left(\frac{dP_l dP_m}{dP_0dP_0} \right) - 1\\
		&  = \left(\frac{1}{n/4}\right)^2 \sum_{l, m = 1}^{n/4} \exp\left(\frac{u_l^{\top}u_m}{\sigma^2}\right) - 1 \\
		& = \left(\frac{4}{n}\right)^2 \left[\sum_{l = 1}^{n/4}\bigl\{\exp\bigl(c \log(n)\bigr)\bigr\} + (n/4)(n/4 - 1)\right] - 1\\&  = 4n^{-1}(n^c - 1),
	\end{align*}
	where the third identity follows from the observation that for $l, m  =1, \ldots, n/4$, 
	\[
		u_l^{\top}u_m = \mathbbm{1}\{l = m\} c \sigma^2\log(n).
	\]

Therefore for any $0 < c < 1$, there exists a sufficiently large $n(c)$ such that for any $n \geq n(c)$, $4n^{-1}(n^c - 1) \leq 1/16$.  This combining with 
	 \eqref{eq:lecam} and \eqref{eq:deduction} provides the desired result.

\end{proof}

\begin{proof}[Proof of \Cref{lemma-error-opt}]
Let $P_0$ denote the joint distribution of the independent random variables $\{Y_i\}_{i=1}^n$, where
	\[
		Y_1, \ldots, Y_{\Delta}  \stackrel{i.i.d.}{\sim}  \mathcal{N}(0, \sigma^2)  \quad \text{and} \quad Y_{\Delta + 1}, \ldots, Y_n  \stackrel{i.i.d.}{\sim}  \mathcal{N}(\kappa, \sigma^2);
	\]
	and, similarly, let $P_1$ be the joint distribution of the independent random variables $\{Z_i\}_{i = 1}^n$ such that
	\[
		Z_1, \ldots, Z_{\Delta + \delta}  \stackrel{i.i.d.}{\sim}  \mathcal{N}(0, \sigma^2), \quad \text{and} \quad Z_{\Delta + \delta + 1}, \ldots, Z_n  \stackrel{i.i.d.}{\sim}  \mathcal{N}(\kappa, \sigma^2),
	\]
	where $\delta$ is a positive integer no larger than $n-1 - \Delta$.  Observe that $\eta(P_0) = \Delta$ and $\eta(P_1) = \Delta + \delta$.  By Le Cam's Lemma \citep[e.g.][]{yu1997assouad} and Lemma~2.6 in \cite{Tsybakov2009}, it holds that
	\[
		\inf_{\hat \eta} \sup_{P\in \mathcal{Q}} \mathbb{E}_P\bigl(|\hat \eta - \eta|\bigr)  \geq \delta\bigl\{1- d_{\mathrm{TV}}(P_0, P_1)\bigr\} \geq \frac{\delta}{2}\exp\left(-KL(P_0, P_1)\right), 
		% \geq \frac{\delta}{2} \exp\left(-\frac{\kappa^2\delta}{\sigma^2}\right).
	\]
	where $KL(\cdot, \cdot)$ is the Kullback–Leibler divergence between two probability measures.
	
Since both $P_0$ and $P_1$ are product measures, it holds that 
	\[
		KL( P_0,P_1) =  \sum_{i \in \{\Delta + 1, \ldots, \Delta + \delta\} } KL(P_{0, i}, P_{1, i}) = \delta \frac{\kappa^2}{\sigma^2},
	\]	
	where $P_{0, i}$ and $P_{1, i}$ are the distributions of $Y_i$ and $Z_i$, respectively and the last identity follows from the fact that, if $P$ and $Q$ are the normal distributions with common variance $\sigma^2$ and means $\mu_1$ and $\mu_2$, respectively, then $K(P,Q) = \frac{ (\mu_1 - \mu_2 )^2}{\sigma^2}$. Thus,
	\begin{equation}\label{eq:second.lower}
		\inf_{\hat \eta} \sup_{P\in \mathcal Q^n} \mathbb{E}_P\bigl(|\hat \eta - \eta|\bigr) \geq  \frac{\delta}{2}\exp\left(-\delta\frac{\kappa^2}{\sigma^2}\right).
	\end{equation}	
	
Next, set $\delta = \min \{ \lceil \frac{\sigma^2}{\kappa^2} \rceil, n - 1 - \Delta\}$. 
By the assumption on $\zeta_n$, for all $n$ large enough we must have that $\delta = \lceil \frac{\sigma^2}{\kappa^2} \rceil.$
Indeed, if $ n - 1 - \Delta \leq \lceil \frac{\sigma^2}{\kappa^2} \rceil$ then, as $\Delta < n/2$, we must have that $\frac{\kappa^2}{\sigma^2} \leq \frac{1}{n-2-\Delta} < \frac{1}{n/2 - 2}$,
and, therefore, that
\[
\frac{\kappa^2 \Delta }{\sigma^2} < \frac{\Delta}{n/2-2} < \frac{n}{n/2 - 2} < 10,
\]
where we may assume that $n >4$. Since $\frac{\kappa^2 \Delta}{\sigma^2} \geq \zeta^2_n$ by assumption and $\zeta_n$ is diverging as $n \rightarrow \infty$, the above bound can only hold for finitely many $n$. The claimed bound now follows from \eqref{eq:second.lower}, for all $n$ large enough. 
\end{proof}

\section{Proofs of the Results in  \Cref{section:potts}}
\label{sec:app-1}
%!TEX root = ./draft.tex

In this section, we provide technical details of the proof of \Cref{prop:1d localization}.  Recalling \Cref{assume:change}, for any change point $\eta_k$, observe that the interval $I = \{\eta_{k-1} + 1, \ldots, \eta_k\}$ contains one change point, but the signal $\{f_i\}_{i=1}^n$ is unchanged in $I$.  For convenience, in this section, any interval $I$ is said to contain a true change point if there exists $k \in \{1, \ldots, K\}$ such that $\{\eta_{k}, \eta_{k}+1\} \subset I$, where $|I| \geq 2$.  This convention ensures that if $I$ contains a true change point, then it is necessary that there exist $i, j \in I$ satisfying $f_i \neq f_j$. 

\begin{lemma}\label{lemma-pre}
Let $I_1$ and $I_2$ denote any two disjoint intervals of $\{1, \ldots, n\}$ and $I = I_1 \cup I_2$.  For any sequences $\{X_i\}_{i=1}^n, \{Y_i\}_{i=1}^n \subset \mathbb{R}$, it holds that
	\begin{equation}\label{eq-lemma2-1}
		\sum_{i \in I} \bigl(Y_i - \overline Y_{I}\bigr)^2 = \sum_{i \in I_1} \bigl(Y_i - \overline Y_{I_1}\bigr)^2 + \sum_{i \in I_2}\bigl(Y_i - \overline Y_{I_2}\bigr)^2 + \frac{|I_1| |I_2|}{|I_1| +|I_2|}  \bigl(\overline Y_{I_1} - \overline Y_{I_2}\bigr)^2,
	\end{equation}
	and 
	\begin{align}
	& \sum_{i \in I} \bigl(X_i - \overline X_{I}\bigr)\bigl(Y_i - \overline Y_{I}\bigr) \nonumber \\
	= & \sum_{i \in I_1}\bigl(X_i - \overline X_{I_1}\bigr)\bigl(Y_i -\overline Y_{I_1}\bigr) + \sum_{i \in I_2}\bigl(X_i - \overline X_{I_2}\bigr)\bigl(Y_i - \overline Y_{I_2}\bigr) + \frac{| I_1 | | I_2| }{|I_1| +|I_2|} \bigl(\overline X_{I_1} -\overline X_{I_2}\bigr) \bigl(\overline Y_{I_1} -\overline Y_{I_2}\bigr). \label{eq-lemma2-2}
	\end{align}
\end{lemma}

\begin{proof}
Without loss of generality, let $I_1 = \{1, \ldots, n_1\}$ and $I_2 = \{n_1 + 1, \ldots, n = n_1 + n_2\}$.  For simplicity, denote $\overline{X} = \overline{X}_I$, $\overline{X}_1 = \overline{X}_{I_1}$ and $\overline{X}_2 = \overline{X}_{I_2}$.  The results \eqref{eq-lemma2-1} and \eqref{eq-lemma2-2} can be proved by similar arguments.  We will only show \eqref{eq-lemma2-2} here.

Observe that
	\begin{align*}
	& \sum_{i = 1}^{n} \bigl(X_i - \overline X\bigr)\bigl(Y_i - \overline Y\bigr)  = \sum_{i = 1}^{n_1} \left\{X_i - \overline{X}_1 + \frac{n_2\bigl(\overline{X}_1 - \overline{X}_2\bigr)}{n_1 + n_2}\right\}\left\{Y_i - \overline{Y}_1 + \frac{n_2\bigl(\overline{Y}_1 - \overline{Y}_2\bigr)}{n_1 + n_2}\right\} \\
	+ & \sum_{i = n_1 + 1}^{n} \left\{X_i - \overline{X}_2 - \frac{n_1\bigl(\overline{X}_1 - \overline{X}_2\bigr)}{n_1 + n_2}\right\}\left\{Y_i - \overline{Y}_2 - \frac{n_1\bigl(\overline{Y}_1 - \overline{Y}_2\bigr)}{n_1 + n_2}\right\} \\
	 = & \sum_{i = 1}^{n_1} \bigl(X_i - \overline X_1\bigr)\bigl(Y_i -\overline Y_1\bigr) + \sum_{i = n_1 + 1}^{n_2} \bigl(X_i - \overline X_2\bigr)\bigl(Y_i -\overline Y_2\bigr) + \frac{n_1n_2}{n_1 + n_2}\bigl(\overline{X}_1 - \overline{X}_2\bigr)\bigl(\overline{Y}_1 - \overline{Y}_2\bigr).
	\end{align*}

\end{proof}

\begin{lemma}\label{lem:max}
Assume that the sequence $\{Y_i\}_{i = 1}^n \subset \mathbb{R}$ satisfies \Cref{assume:change}.  It holds that 
	\begin{align*}
	\mathbb{P}\biggl\{\sup_{1\le a < b< c\le n} \sqrt{\frac{(b-a)(c-b)}{c-a}} \bigl|\overline Y_{(a+1,b]} - \overline f_{(a+1,b]} + \overline Y_{(b+1,c]} - \overline f_{(b+1,c]} \bigr| \leq C_{\mathcal{B}}\sigma \sqrt{\log(n)} \biggr\} \geq e\cdot n^{3-c_{\mathcal{B}}},
	\end{align*}
	where $c_{\mathcal{B}}$ is an absolute constant chosen to satisfy $c_{\mathcal{B}} > 3$ and $C_{\mathcal{B}} > 0$ only depends on $c_{\mathcal{B}}$.
\end{lemma}

\begin{proof}
It follows from \Cref{assume:change} that for all $i \in \{1, \ldots, n\}$, $Y_i - f_i$ is a 	centred sub-Gaussian random variable with $\max_i\|Y_i - f_i\|_{\psi_2} \leq \sigma$.  Due to Hoeffding inequality \citep[see e.g.][]{vershynin2010introduction}, it holds that for any non-empty set $I \subset \{1, \ldots, n\}$ and any $\varepsilon > 0$,
	\[
	\mathbb{P}\bigl\{\bigl|\overline{Y}_I - \overline{f}_I\bigr| > \varepsilon \bigr\} \leq e \cdot \exp\left(- \frac{c|I|\varepsilon^2}{\sigma^2}\right),
	\] 
	and for any triple $i_1 < i_2 < i_3$ chosen in $\{1, \ldots, n\}$
	\begin{align*}
	\mathbb{P}\left\{\sqrt{\frac{(i_2-i_1)(i_3-i_2)}{i_3-i_1}} \bigl|\overline Y_{(i_1+1, i_2]} - \overline f_{(i_1+1, i_2]} + \overline Y_{(i_2+1, i_3]} - \overline f_{(i_2+1, i_3]} \bigr| \geq \varepsilon \right\} \leq  e \cdot \exp \left(-\frac{c\varepsilon^2}{\sigma^2}\right),
	\end{align*}
	where $c > 0$ is an absolute constant only depending on $\sigma$. The result follows from a union bound.
\end{proof}

For simplicity, in the rest of the proof, we will let $C_{\mathcal{B}} = 1$ and set $c_{\mathcal{B}} > 3$.  This will only affect the constant $C_\lambda$, and in the statement of \Cref{prop:1d localization}, we require $C_{\lambda} > 0$ to be large enough. 

Since the change points of $\widehat{u}$ are our change point estimators, with the error rate 
	\[
		\epsilon_k = C_{\epsilon}\sigma^2\log(n)/\kappa_k^2,
	\]
	we refer to $\eta_k$ as an undetected change point, if $\eta_k \in (s, e] \in \widehat{\mathcal{P}}(\widehat{u})$ and
	\begin{equation}\label{eq-appa-1}
	\epsilon_k - s = \epsilon_k - \epsilon_{k-1} - (s - \epsilon_{k-1}) \geq \Delta - C_{\epsilon}\sigma^2 \log(n)/\kappa_k^2 > \Delta/3,
	\end{equation}
	and similarly $e - \epsilon_k > \Delta/3$.  The first and second inequalities of \eqref{eq-appa-1} follow from Assumptions~\ref{assume:change} and \ref{assum-phase}, respectively.   In the rest of this section, let $\widehat{\mathcal{P}} = \widehat{\mathcal{P}}\bigl(\widehat{u}\bigr)$.

\subsection{Step 1: no more than two true change points}
\label{section:three changes}

In order to show that no $I \in \widehat{\mathcal{P}}$ contains more than two true change points, it suffices to show that no $I \in \widehat{\mathcal{P}}$ contains undetected change points, due to the minimal spacing $\Delta$ condition in \Cref{assume:change}.

\begin{lemma}\label{lemma:three changes} 
Let $\{Y_i\}_{i=1}^n$ satisfy Assumptions~\ref{assume:change} and \ref{assum-phase}, and $\lambda$ satisfy the condition
	\begin{equation}\label{eq-lemmma4-lambda}
		\sigma^2 \log(n) \le\lambda \le  \kappa^2 \Delta /48.
	\end{equation}
	Then, in the event $\mathcal{B}$, it holds that no $I \in \widehat{\mathcal{P}}$ contains any undetected change point.
\end{lemma}

\begin{proof} 
We first point out that due to \Cref{assum-phase}, \eqref{eq-lemmma4-lambda} is not an empty set.

For the sake of contradiction, suppose that there exists $I \in \widehat{\mathcal{P}}$ containing an undetected change point $\eta_k$, i.e.,
	\begin{equation}\label{eq-lemma4-1}
		\min\{e-(\eta_k+1), \eta_k-s \} > \Delta /3.
	\end{equation}
	
Denote 	
	\begin{align*}
		I_1 = (s, \eta_{k} - \Delta/3], \quad  I_2 = (\eta_{k} -\Delta/3, \eta_k], \quad  I_3 = (\eta_k, \eta_{k} + \Delta/3], \quad \text{and} \quad I_4 = (\eta_{k} + \Delta/3 , e],
	\end{align*}
	none of which is empty due to \eqref{eq-lemma4-1}.

Let $\widetilde{\mathcal{P}}$ be such that
	\[
		\widetilde{\mathcal{P}} = \widehat{\mathcal{P}} \cup \{I_1, I_2, I_3, I_4\} \setminus \{I\},
	\]
	and $\widetilde{u}$ be the piecewise constant vector induced by $\widetilde{\mathcal{P}}$.  By the definition of $\widehat  u$, it holds that
	\[
		H\bigl(\widehat{u}, \{Y_i\}_{i=1}^n, \lambda\bigr) \leq H\bigl(\widetilde{u}, \{Y_i\}_{i=1}^n, \lambda\bigr).
	\]
	
Since $\widetilde{\mathcal{P}}$ is a refinement of $\widehat{\mathcal{P}}$ and we have assumed in \Cref{assume:change} that the distributions of $Y_i$'s have continuous density functions, it follows that 
	\[
		\lambda( \bigl\|D\widehat{u}\bigr\|_0 - \bigl\|D\widetilde{u}\bigr\|_0 )= -3\lambda.
	\]
	Then
	\begin{align}
		0 & \geq H\bigl(\widehat{u}, \{Y_i\}_{i=1}^n, \lambda\bigr) - H\bigl(\widetilde{u}, \{Y_i\}_{i=1}^n, \lambda\bigr) \nonumber \\
		& = -3\lambda + \sum_{i\in I } (Y_i - \overline Y_{I})^2 - \sum_{i\in I_1 } (Y_i - \overline Y_{I_1})^2 - \sum_{i\in I_2 } (Y_i - \overline Y_{I_2})^2-\sum_{i\in I_3 } (Y_i - \overline Y_{I_3})^2-\sum_{i\in I_4 } (Y_i - \overline Y_{I_4})^2 \nonumber \\
		& \geq -3\lambda + \frac{| I_2 | | I_3| }{|I_2| +|I_3|}  (\overline Y_{I_2} -\overline Y_{I_3})^2\nonumber \\
		& = -3\lambda +  \frac{| I_2 | | I_3| }{|I_2| +|I_3|}\bigl\{(\overline Y_{I_2} - f_{\eta_{k}}) - (\overline Y_{I_3} - f_{\eta_{k+1}}) + (f_{\eta_{k}} - f_{\eta_{k+1}}) \bigr\}^2\nonumber \\
		& \geq -3\lambda + \frac{\Delta}{12} (f_{\eta_{k}} - f_{\eta_{k+1}})^2 -  \frac{| I_2 | | I_3| }{|I_2| +|I_3|}\bigl\{(\overline Y_{I_2} - f_{\eta_{k}}) - (\overline Y_{I_3} - f_{\eta_{k+1}})\bigr\}^2 \nonumber \\
		& \geq -4\lambda + \frac{\Delta}{12}\kappa_k^2 \nonumber \\
		& > 0, \label{eq-lemma4-contra}
	\end{align}
	where the second inequality follows from \eqref{eq-lemma2-1} by first splitting $I = \{I_1, I_2, I_3\} \cup \{I_4\}$, then $\{I_1, I_2, I_3\} = \{I_1\} \cup \{I_2, I_3\}$ and $\{I_2, I_3\} = \{I_2\} \cup \{I_3\}$; the third inequality follows from the observation that $(x + y)^2 \geq x^2/2 - y^2$ and letting $x = f_{\eta_{k}} - f_{\eta_{k+1}}$, $y = (\overline Y_{I_2} - f_{\eta_{k}}) - (\overline Y_{I_3} - f_{\eta_{k+1}})$; the fourth inequality follows from the definition of $\mathcal{B}$ and \eqref{eq-lemmma4-lambda}; and the last inequalities is due to \eqref{eq-lemmma4-lambda}. 
	
Since \eqref{eq-lemma4-contra} is a contradiction, we conclude that there is no interval containing undetected change point. 

\end{proof}

\subsection{Step 2: exactly two true change points}\label{section:two changes}

\begin{lemma}\label{lemma:two change localization}
Let $\{Y_i\}_{i=1}^n$ satisfy Assumptions~\ref{assume:change} and \ref{assum-phase} and set $\lambda = C_{\lambda}\sigma^2\log(n)$,  where $C_{\lambda} > 1$. In the event $\mathcal{B}$, it holds that if $I = (s, e] \in \widehat{\mathcal{P}}$ contains exactly two change points, say $\eta_k$ and $\eta_{k+1}$, then 
	\[
		\eta_k - s + 1 \leq 12\lambda/\kappa_k^2, \mbox{ and } e -\eta_{k+1} \leq 12\lambda/\kappa_{k+1}^2.
	\]
\end{lemma}

\begin{proof}
Let $I_1 = (s, \eta_k]$, $I_2 = (\eta_k, \eta_{k+1}]$ and $I_3 = (\eta_{k+1}, e]$.  Since $I$ contains exactly two true change points, none of $I_1$, $I_2$ or $I_3$ is an empty set, and $\{f_i\}_{i=1}^n$ is constant on $I_1$, $I_2$ and $I_3$.  Denote by $\widehat{u}$ the solution of \eqref{eq-H} with inputs $\{Y_i\}_{i=1}^n$ and $\lambda$, and by $\widehat{\mathcal{P}}$ the interval partition induced by $\widehat{u}$.

Let $\widetilde{\mathcal{P}}_1$ and $\widetilde{\mathcal{P}}_2$ be
	\[
		\widetilde{\mathcal{P}}_1 = \widehat{\mathcal{P}} \cup \{I_1, I_2 \cup I_3\} \setminus \{I\}, \, \mbox{ and } \widetilde{\mathcal{P}}_2 = \widehat{\mathcal{P}} \cup \{I_1, I_2, I_3\} \setminus \{I\},
	\]
	respectively; and let $\widetilde{u}_1$ and $\widetilde{u}_2$ be the piecewise-constant vectors induced by $\widetilde{\mathcal{P}}_1$ and $\widetilde{\mathcal{P}}_2$, respectively.  

It follows from \Cref{lemma-pre} that 
	\[
		H\bigl(\widetilde{u}_1, \{Y_i\}_{i=1}^n, \lambda\bigr) - H\bigl(\widehat{u}, \{Y_i\}_{i=1}^n, \lambda\bigr) = \lambda - \frac{|I_1| |I_2 \cup I_3|}{|I_1| + |I_2 \cup I_3|}\bigl(\overline Y_{I_1} -\overline Y_{I_2\cup I_3}\bigr)^2
	\]
	and
	\[
		H\bigl(\widetilde{u}_2, \{Y_i\}_{i=1}^n, \lambda\bigr) - H\bigl(\widetilde{u}_1, \{Y_i\}_{i=1}^n, \lambda\bigr) =  \lambda - \frac{|I_2| |I_3|}{|I_2| +|I_3|}\bigl(\overline Y_{I_2} - \overline Y_{I_3}\bigr)^2.
	\]
	Then,
	\begin{align*}
	0 & \leq H\bigl(\widetilde{u}_2, \{Y_i\}_{i=1}^n, \lambda\bigr) -  H\bigl(\widehat{u}, \{Y_i\}_{i=1}^n, \lambda\bigr)\\
	&  \leq 2\lambda -  \frac{|I_2| |I_3|}{|I_2| +|I_3|}\bigl(\overline Y_{I_2} - \overline Y_{I_3}\bigr)^2 \\
	& \leq 2\lambda - \frac{1}{2}\frac{|I_2| |I_3|}{|I_2| +|I_3|}\bigl\{\bigl(f_{\eta_{k+1}} - f_{\eta_{k+2}}\bigr)^2 - 2\bigl(\overline Y_{I_2} - f_{\eta_{k+1}} - \overline Y_{I_3} + f_{\eta_{k+2}}\bigr)^2\bigr\} \\
	& \leq 2 \lambda -  \frac{1}{2}\frac{| I_2 | |  I_3| }{|I_2| +|I_3|}  \kappa_{k+1}^2 + \lambda,
	\end{align*}
	where the third inequality uses the same argument in the third inequality of \eqref{eq-lemma4-contra}, and the last inequality follows from the definition of $\mathcal{B}$ and \Cref{assume:change}.

If $|I_2| \leq |I_3|$, then
	\[
		\Delta/2 \leq |I_2|/2 \leq \frac{| I_2 | |  I_3| }{|I_2| +|I_3|}   \leq 6\lambda/\kappa_{k+1}^2,
	\]
	which contradicts \Cref{assum-phase}.  Therefore it must hold that $|I_2| > |I_3|$, which implies 
	\[
		|I_3|/2 \leq \frac{| I_2 | |  I_3| }{|I_2| +|I_3|} \leq 6\lambda/\kappa_{k+1}^2.
	\]
	Then $e - \eta_{k+1} \leq 12\lambda/\kappa_{k+1}^2$.  It can be shown similarly that $\eta_k - s + 1 \leq 12\lambda/\kappa_k^2$.
	
\end{proof}

\subsection{Step 3: one and only one change point}
\label{section:one change}

Let $I_1 = (s, e_1] \in \widehat{\mathcal{P}}$ contain exactly one true change point, namely $\eta_k$.  With our convention set at the beginning of \Cref{sec:app-1}, it holds that
	\begin{equation}\label{eq-one-change-se1}
		\eta_{k-1}+1 \leq s \leq \eta_k < \eta_{k} + 1 \leq e_1 \leq \eta_{k+1}.
	\end{equation}
	Denote $\delta = e_1 - \eta_{k}$ and $\epsilon = \eta_{k} - (s-1)$.  Without loss of generality, we assume that
	\begin{align}\label{eq:symmetry}
		0 < \epsilon \le \delta. 
	\end{align}
	We are to show that there exists an absolute constant $C > 8$ such that 	
	\begin{equation}\label{eq:one change point localization-1}
		\epsilon = |\eta_k- s + 1| \leq C\lambda/\kappa_k^2
	\end{equation}
	and
	\begin{equation}\label{eq:one change point localization-2}
		\epsilon_1 = |\eta_{k+1} - e_1| \leq C\lambda/\kappa_{k+1}^2	.
	\end{equation}

Equation~\eqref{eq:one change point localization-1} will be shown in \Cref{lemma:one side of one change point}.  To show \eqref{eq:one change point localization-2}, we rely on the following arguments (see \Cref{fig-l0-step3} for an illustration):
	\begin{itemize}
	\item [(i)] Let  $I_2 = (e_1, e_2]$ be the interval to the immediate right of $I_1$ in $\widehat{\mathcal{P}}$.  It must hold that
		\begin{equation}\label{eq-one-change-e1e2}
		e_1 \leq \eta_{k+1} < \eta_{k+1} + 1 \leq e_2.
		\end{equation}
		This will be shown in \Cref{lemma:one change then no change}.
	\item [(ii)] It follows from \Cref{section:three changes} that there are at most two true change points in $(e_1, e_2]$.  If there are exactly two true change points, then due to \Cref{section:two changes}, \eqref{eq:one change point localization-2} holds. 
	\item [(iii)] If $e_2 \leq \eta_{k+2}$, then we let $\epsilon_1 = \eta_{k+1} - (e - 1)$ and $\delta_1 = e_2 - \eta_{k+1}$. \Cref{lemma-eps-delta} shows that $\delta_1 < \epsilon_1$ is impossible. Thus, $\epsilon_1 \leq \delta_1$ and we then rely on \Cref{lemma:one side of one change point}. 
	\end{itemize}

\vskip 3mm
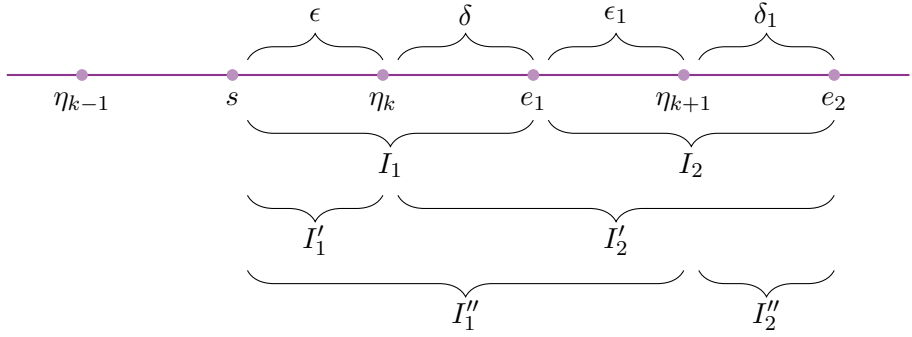
\begin{figure}
\begin{center}
\begin{tikzpicture}
\draw[Fuchsia!100, thick]  (0, 0) -- (12, 0);

\filldraw [Fuchsia!50] (1,0) circle (2pt);	
\draw (1, -0.1) node[below] {$\eta_{k-1}$};
\filldraw [Fuchsia!50] (3,0) circle (2pt);	
\draw (3, -0.1) node[below] {$s$};
\filldraw [Fuchsia!50] (5,0) circle (2pt);	
\draw (5, -0.1) node[below] {$\eta_k$};
\filldraw [Fuchsia!50] (7,0) circle (2pt);	
\draw (7, -0.1) node[below] {$e_1$};
\filldraw [Fuchsia!50] (9,0) circle (2pt);	
\draw (9, -0.1) node[below] {$\eta_{k+1}$};
\filldraw [Fuchsia!50] (11,0) circle (2pt);	
\draw (11, -0.1) node[below] {$e_2$};

\draw [decorate,decoration={brace,amplitude=10pt},xshift=0pt,yshift=0pt]
(3.2,0.2) -- (5, 0.2) node [black,midway,xshift=0cm, yshift = 0.6cm] {$\epsilon$};
\draw [decorate,decoration={brace,amplitude=10pt},xshift=0pt,yshift=0pt]
(5.2,0.2) -- (7, 0.2) node [black,midway,xshift=0cm, yshift = 0.6cm] {$\delta$};
\draw [decorate,decoration={brace,amplitude=10pt},xshift=0pt,yshift=0pt]
(7.2,0.2) -- (9, 0.2) node [black,midway,xshift=0cm, yshift = 0.6cm] {$\epsilon_1$};
\draw [decorate,decoration={brace,amplitude=10pt},xshift=0pt,yshift=0pt]
(9.2,0.2) -- (11, 0.2) node [black,midway,xshift=0cm, yshift = 0.6cm] {$\delta_1$};

\draw [decorate,decoration={brace,amplitude=10pt, mirror},xshift=0pt,yshift=0pt]
(3.2, -0.6) -- (7, -0.6) node [black,midway,xshift=0cm, yshift = -0.6cm] {$I_1$};
\draw [decorate,decoration={brace,amplitude=10pt, mirror},xshift=0pt,yshift=0pt]
(7.2, -0.6) -- (11, -0.6) node [black,midway,xshift=0cm, yshift = -0.6cm] {$I_2$};

\draw [decorate,decoration={brace,amplitude=10pt, mirror},xshift=0pt,yshift=0pt]
(3.2, -1.6) -- (5, -1.6) node [black,midway,xshift=0cm, yshift = -0.6cm] {$I_1'$};
\draw [decorate,decoration={brace,amplitude=10pt, mirror},xshift=0pt,yshift=0pt]
(5.2, -1.6) -- (11, -1.6) node [black,midway,xshift=0cm, yshift = -0.6cm] {$I_2'$};

\draw [decorate,decoration={brace,amplitude=10pt, mirror},xshift=0pt,yshift=0pt]
(3.2, -2.6) -- (9, -2.6) node [black,midway,xshift=0cm, yshift = -0.6cm] {$I_1''$};
\draw [decorate,decoration={brace,amplitude=10pt, mirror},xshift=0pt,yshift=0pt]
(9.2, -2.6) -- (11, -2.6) node [black,midway,xshift=0cm, yshift = -0.6cm] {$I_2''$};

\end{tikzpicture}
\end{center}
\caption{Illustrations of the interval constructions used in the Step 3 in the proof of \Cref{prop:1d localization}. \label{fig-l0-step3}}
\end{figure}

\begin{lemma}\label{lemma:one side of one change point}
Let $\{Y_i\}_{i=1}^n$ satisfy Assumptions~\ref{assume:change} and \ref{assum-phase} and set $\lambda \geq C_{\lambda}\sigma^2\log(n)$, with $C_{\lambda} \geq 1$. In the event $\mathcal{B}$, it holds that if $I_1 = (s, e_1] \in \widehat{\mathcal{P}}$ contains exactly one change point, say $\eta_k$, then
	\[
		\min\bigl\{|J_1|,\, |J_2|\bigr\} \leq  8\lambda/\kappa_k^2,
	\]
	where $J_1 = (s, \eta_k]$ and $J_2 = (\eta_k, e_1]$.
\end{lemma}

\begin{proof}
Observe that neither $J_1$ nor $J_2$ is empty by definition, and that $\{f_i\}_{i=1}^n$ is constant within $J_1$ and $J_2$, respectively.  Let $\widetilde{\mathcal{P}}$ be such that
	\[
		\widetilde{\mathcal{P}} = \widehat{\mathcal{P}} \cup \{J_1, J_2\} \setminus \{I_1\},
	\]
	and let $\widetilde{u}$ be the piecewise-constant vector induced by $\widetilde{\mathcal{P}}$.  
	
Recall that $\mathbb{E}(\overline{Y}_{J_1}) = f_{\eta_k}$ and $\mathbb{E}(\overline{Y}_{J_2}) = f_{\eta_{k+1}}$.  Without loss of generality, assume $f_{\eta_{k+1}} = f_{\eta_k} + \kappa_k$.  Thus,
	\begin{align*}
		0 & \geq H\bigl(\widehat{u}, \{Y_i\}_{i=1}^n, \lambda\bigr) - H\bigl(\widetilde{u}, \{Y_i\}_{i=1}^n, \lambda\bigr) \\
		& = -\lambda + \sum_{i\in I_1}(Y_i - \overline Y_{I_1})^2 - \sum_{i \in J_1}(Y_i - \overline Y_{J_1})^2 - \sum_{i \in J_2}(Y_i - \overline Y_{I_2})^2 \\
		& = -\lambda + \frac{|J_1| |J_2|}{|I_1|}(\overline Y_{J_1} -\overline Y_{J_2})^2\\
		& = -\lambda + \frac{|J_1| |J_2|}{|I_1|} \bigl\{(\overline Y_{J_1} - f_{\eta_k}) - (\overline Y_{J_2} - f_{\eta_k} - \kappa_k) -\kappa_k \bigr\}^2 \\
		& \geq -\lambda + \frac{|J_1| |J_2|}{2|I_1|} \bigl\{\kappa_k^2 - 2(\overline Y_{J_1} - f_{\eta_k} - \overline Y_{J_2} + f_{\eta_{k+1}})^2 \bigr\} \\
		& \geq -2\lambda + \frac{|J_1| |J_2|}{2|I_1|}\kappa_k^2,
	\end{align*}
	where the second identity follows from \eqref{eq-lemma2-1}, the second inequality from the fact that $(x-y)^2 \geq y^2/2 - x^2$ with $x = \kappa_k$ and $y = (\overline Y_{J_1} - f_{\eta_k}) - (\overline Y_{J_2} - f_{\eta_k} - \kappa_k)$, and the last inequality from the definitions of the event $\mathcal{B}$ and the choice of $\lambda$.  Therefore,	
	\[
		\min\{|J_1|, \, |J_2|\}\kappa^2_k/8 \leq \frac{|J_1| |J_2|}{|I_1|} \kappa_k^2/4 \leq \lambda.
	\]
\end{proof}

\begin{lemma}\label{lemma:one change then no change}
Let $\{Y_i\}_{i=1}^n$ satisfy Assumptions~\ref{assume:change} and \ref{assum-phase} and set $\lambda = C_{\lambda}\sigma^2\log(n)$, with $C_{\lambda} > 85$.  Assume that $I_1 = (s, e_1] \in \widehat{\mathcal{P}}$ contains exactly one change point namely $\eta_k$.  Denote $\delta = e_1 - \eta_{k}$ and $\epsilon = \eta_{k} - (s-1)$.  Assume that $\epsilon \le \delta$.    In the event $\mathcal{B}$, if $I_2 = (e_1, e_2] \in \widehat{\mathcal{P}}$, then it must hold that
	\[
		e_1 \leq \eta_{k+1} < \eta_{k+1} + 1 \leq e_2.
	\] 
\end{lemma}

\begin{proof}
Let $I_1' = (s,  \eta_{k}]$ and $I_2' = (\eta_{k}, e_2]$.  Then $I_1 \cup I_2 = I_1'\cup I_2'$.   Let $\widetilde{\mathcal{P}}_1$ and $\widetilde{\mathcal{P}}_2$ be 
	\[
		\widetilde{\mathcal{P}}_1 = \widehat{\mathcal{P}} \cup \{I_1 \cup I_2\} \setminus \{I_1, I_2\} 
	\]
	and 
	\[
		\widetilde{\mathcal{P}}_2 = \widehat{\mathcal{P}} \cup \{I_1', I_2'\} \setminus \{I_1, I_2\},
	\]
	respectively.  Let $\widetilde{u}_1$ and $\widetilde{u}_2$ be the piecewise-constant vectors induced by $\widetilde{\mathcal{P}}_1$ and $\widetilde{\mathcal{P}}_2$, respectively.

We proceed by contradiction.  We assume that $e_2 \leq \eta_{k+1}$.  Without loss of generality, assume $f_{\eta_{k+1}} = f_{\eta_k} + \kappa_k$.  Due to \Cref{assume:change}, it holds that $\mathbb{E}(\overline Y_{I_1'}) = f_{\eta_k}$, $\mathbb{E}(\overline{Y}_{I_2'}) = f_{\eta_{k+1}} = f_{\eta_k} + \kappa_k$, $\mathbb{E}(\overline{Y}_{I_1}) = f_{\eta_k} + \delta \kappa_k/|I_1|$ and $\mathbb{E}(\overline{Y}_{I_2}) = f_{\eta_{k+1}} = f_{\eta_k} + \kappa_k$.  Then, 
	\begin{align}
		0 & \leq H\bigl(\widetilde{u}_1, \{Y_i\}_{i=1}^n, \lambda\bigr) -H\bigl(\widehat{u}, \{Y_i\}_{i=1}^n, \lambda\bigr) \nonumber \\
		&  = -\lambda+ \frac{|I_1| |I_2|}{|I_1| +|I_2|}(\overline Y_{I_1} -\overline Y_{I_2})^2 \nonumber \\
		& = -\lambda + \frac{|I_1| |I_2|}{|I_1| + |I_2|} \bigl\{\overline Y_{I_1} - \mathbb{E}(\overline Y_{I_1}) - \overline Y_{I_2} + f_{\eta_{k+1}} + (\delta/|I_1| - 1)\kappa_k\bigr\}^2 \nonumber \\
		& \leq - \lambda + \frac{|I_1| |I_2|}{|I_1| +|I_2|}\bigl\{5(\overline Y_{I_1} - \mathbb{E}(\overline Y_{I_1}) - \overline Y_{I_2} + f_{\eta_{k+1}})^2+ \frac{5}{4}(\delta/|I_1| - 1)^2\kappa^2_k\bigr\} \nonumber \\
		& \leq -\lambda + 5\sigma^2\log(n) + \frac{5}{4}\frac{| I_1 | | I_2| }{|I_1| +|I_2|}\frac{\epsilon^2 \kappa_k^2}{|I_1|^2}\label{eq:case 4 combine intervals}
	\end{align}
	where the second inequity follows form the fact that $(x + y)^2 \leq 5 x^2 + (5/4) y^2$ and the last inequality follows from the definition of the event $\mathcal{B}$.  
	
In addition, we have
	\begin{align}
		& H\bigl(\widetilde{u}_2, \{Y_i\}_{i=1}^n, \lambda\bigr) -  H\bigl(\widetilde{u}_1, \{Y_i\}_{i=1}^n, \lambda\bigr) \nonumber\\
		& = \lambda - \frac{|I_1'| |I_2'|}{|I_1'| + |I_2'|}(\overline Y_{I_1'} -\overline Y_{I_2'})^2 \nonumber \\
		& = \lambda - \frac{|I_1'| |I_2'|}{|I_1'| + |I_2'|}\bigl\{\overline Y_{I_1'} - f_{\eta_k} - \overline Y_{I_2'} + f_{\eta_k} + \kappa_k  - \kappa_k\bigr\}^2 \nonumber \\
		& \leq \lambda - \frac{|I_1'| |I_2'|}{|I_1'| + |I_2'|} \left\{\frac{3}{4}\kappa^2_k - 3(\overline Y_{I_1'} - f_{\eta_k} - \overline Y_{I_2'} + f_{\eta_k} + \kappa_k)^2 \right\} \nonumber \\
		& \leq \lambda - \frac{3}{4}\frac{|I_1'| |I_2'|}{|I_1'| + |I_2'|}\kappa^2_k + 3\sigma^2\log(n), \nonumber &\label{eq:case 4 combine intervals-2}
	\end{align}
	where the first inequality follows from the fact that $(x-y)^2 \geq (3/4) y^2 - 4x^2$, and the last inequality follows from the definition of the event $\mathcal{B}$.
	
Then	
	\begin{align*}
		0 & \leq H\bigl(\widetilde{u}_2, \{Y_i\}_{i=1}^n, \lambda\bigr) -H\bigl(\widehat{u}, \{Y_i\}_{i=1}^n, \lambda\bigr) \le  8\sigma^2 \log(n)  +\frac{5}{4}\frac{| I_1 | | I_2| }{|I_1| +|I_2|}\frac{\epsilon^2 \kappa_k^2}{|I_1|^2}- \frac{3}{4}\frac{| I_1' | | I_2'| }{|I_1'| +|I_2'|}\kappa^2_k \\
		& = 8\sigma^2 \log(n) + \frac{\kappa_k^2 \epsilon}{4(|I_1| +|I_2|)} \left\{\frac{5|I_2| \epsilon}{\epsilon + \delta} -3(\delta + |I_2|) \right\} \\
		& \leq 8\sigma^2 \log(n) - \frac{\kappa_k^2 \epsilon}{4(|I_1| +|I_2|)} (3\delta + |I_2|/2),
	\end{align*}
	therefore
	\[
		\kappa_k^2\epsilon \leq 64 \sigma^2 \log(n). 
	\]
Combined  with \eqref{eq:case 4 combine intervals}, this implies that 
	\begin{align*}
		\lambda \leq 5 \sigma^2 \log(n) + 64 \sigma^2\log(n) \frac{5|I_2| \epsilon}{4(|I_1|+|I_2|)|I_1|} \leq 85 \sigma^2 \log(n),
	\end{align*}
	which contradicts with the assumption that $\lambda > 85\sigma^2\log(n)$. 
\end{proof}

\begin{lemma}\label{lemma-eps-delta}	
Let $\{Y_i\}_{i=1}^n$ satisfy Assumptions~\ref{assume:change} and \ref{assum-phase} and set $\lambda = C_{\lambda}\sigma^2\log(n)$ with a sufficiently large $C_{\lambda} > 0$.  Assume that there exists an interval partition $\mathcal{P}$ with induced piecewise constant vector $u$ such that $I_1 = (s, e_1] \in \mathcal{P}$ and $I_2 = (e_1, e_2] \in \mathcal{P}$, where $I_1$ and $I_2$ satisfy \eqref{eq-one-change-se1} and \eqref{eq-one-change-e1e2}.  Let $\epsilon = \eta_k - s + 1$, $\delta = e_1 - \eta_k +1$, $\epsilon_1 = \eta_{k+1} - e + 1$ and $\delta_1 = e_2 - \eta_{k+1} +1$.  Assume $\epsilon < \delta$ and $\epsilon_1 > \delta_1$.  Then in the event $\mathcal{B}$, $u$ is not a minimizer of \eqref{eq-H}.
\end{lemma}

\begin{proof}
For notational simplicity, let
	\[
		f_{\eta_k} = \mu + \omega_1, \quad f_{\eta_{k+1}} = \mu \quad \text{and}  \quad f_{\eta_{k+2}} = \mu + \omega_2.
	\]
	Let $I_1' = (s,  \eta_{k}]$, $I_2' = (\eta_{k}, e_2]$, $I_1'' = (s, \eta_{k+1}]$ and $I_2'' = (\eta_{k+1}, e_2]$.  Then $I_1 \cup I_2 = I_1' \cup I_2' = I_1'' \cup I_2''$.  Let $\widetilde{\mathcal{P}}_1$, $\widetilde{\mathcal{P}}_2$ and $\widetilde{\mathcal{P}}_3$ be such that
	\begin{align*}
		\widetilde{\mathcal{P}}_1 = \mathcal{P} \cup \{I_1 \cup I_2\} \setminus \{I_1, \,I_2\},\, \widetilde{\mathcal{P}}_2 = \mathcal{P} \cup \{I_1', \, I_2'\} \setminus \{I_1, \, I_2\}, \mbox{ and }\widetilde{\mathcal{P}}_3 =  \mathcal{P} \cup \{I_1'',\, I_2''\} \setminus \{I_1,\, I_2\}.
	\end{align*}
	Let $\widetilde{u}_1$, $\widetilde{u}_2$ and $\widetilde{u}_3$ be the piecewise constant vectors induced by $\widetilde{\mathcal{P}}_1$, $\widetilde{\mathcal{P}}_2$ and $\widetilde{\mathcal{P}}_3$.  The population means are
	\[
		\mathbb{E}(\overline{Y}_{I_1}) = \mu + \frac{\epsilon \omega_1}{\epsilon + \delta} \mbox{ and } \mathbb{E}(\overline{Y}_{I_2}) = \mu + \frac{\delta_1 \omega_2}{\epsilon_1 + \delta_1}.
	\]

Let $0<\alpha<1$ be a fixed constant to be specified later.  We have the following:
	\begin{align*}
		& H\bigl(\widetilde{u}_1, \{Y_i\}_{i=1}^n, \lambda\bigr) - H\bigl(u, \{Y_i\}_{i=1}^n, \lambda\bigr) = -\lambda + \frac{|I_1| |I_2|}{|I_1| + |I_2|} (\overline Y_{I_1} - \overline Y_{I_2})^2 \\
		= & -\lambda+ \frac{| I_1 | | I_2| }{|I_1| +|I_2|} \bigl\{\overline Y_{I_1} - \mathbb{E}(\overline Y_{I_1}) - \overline Y_{I_2} + \mathbb{E}(\overline Y_{I_2}) + \mathbb{E}(\overline Y_{I_1}) - \mathbb{E}(\overline Y_{I_2})\bigr\}^2 \\
		\leq & -\lambda + 2(1+\alpha)\alpha^{-1}\sigma^2 \log(n) + (1+\alpha)\frac{|I_1| |I_2|}{|I_1| +|I_2|} \left(\frac{\epsilon \omega_1}{\epsilon + \delta} - \frac{\delta_1\omega_2}{\epsilon_1 + \delta_1}\right)^2, \\
		& H\bigl(\widetilde{u}_2, \{Y_i\}_{i=1}^n, \lambda\bigr) - H\bigl(\widetilde{u}_1, \{Y_i\}_{i=1}^n, \lambda\bigr) = \lambda - \frac{|I_1'| |I_2'|}{|I_1'| + |I_2'|}(\overline Y_{I_1'} - \overline Y_{I_2'})^2 \\
		= & \lambda - \frac{|I_1'| |I_2'|}{|I_1'| + |I_2'|}\bigl\{\overline Y_{I_1'} - \mathbb{E}(\overline Y_{I_1'}) -\overline Y_{I_2'} + \mathbb{E}(\overline Y_{I_2'}) + \mathbb{E}(\overline Y_{I_1'}) - \mathbb{E}(\overline Y_{I_2'}) \bigr\}^2 \\
		\leq & \lambda - (1-\alpha)\frac{|I_1'| |I_2'|}{|I_1'| + |I_2'|} \left(\omega_1 - \frac{\omega_2 \delta_1}{\delta + \epsilon_1 + \delta_1}\right)^2 + \frac{2(1-\alpha)}{\alpha}\sigma^2\log(n),
	\end{align*}
	and
	\begin{align*}
		& H\bigl(\widetilde{u}_3, \{Y_i\}_{i=1}^n, \lambda\bigr) - H\bigl(\widetilde{u}_1, \{Y_i\}_{i=1}^n, \lambda\bigr)  = \lambda - \frac{|I_1''| |I_2''|}{|I_1''| + |I_2''|}(\overline Y_{I_1''} - \overline Y_{I_2''})^2 \\
		= & \lambda - \frac{|I_1''| |I_2''|}{|I_1''| + |I_2''|}\bigl\{\overline Y_{I_1''} - \mathbb{E}(\overline Y_{I_1''}) -\overline Y_{I_2''} + \mathbb{E}(\overline Y_{I_2''}) + \mathbb{E}(\overline Y_{I_1''}) - \mathbb{E}(\overline Y_{I_2''}) \bigr\}^2 \\
		\leq & \lambda - (1-\alpha)\frac{|I_1''| |I_2''|}{|I_1''| + |I_2''|} \left(\omega_2 - \frac{\omega_1 \epsilon}{\epsilon + \delta + \epsilon_1}\right)^2 + \frac{2(1-\alpha)}{\alpha}\sigma^2\log(n).
	\end{align*}

For the rest of the proof, we proceed by contradiction by assuming that $u$ is the minimizer of \eqref{eq-H}.  We will consider the cases $\omega_1\omega_2 > 0$ and $\omega_1\omega_2 < 0$ separately in {\bf Steps 1} and {\bf 2}, respectively.

\vskip 3mm
\noindent {\bf Step 1.} Suppose $\omega_1\omega_2 > 0$.  Without loss of generality, assume $\omega_1, \omega_2 > 0$ and for some $0 < \beta \leq 1$, it holds that
	\[
		\frac{\delta_1\omega_2}{\epsilon_1 + \delta_1} = \beta \frac{\epsilon \omega_1}{\epsilon + \delta}.
	\]

We have
    \begin{align}\label{eq:one change point  case 3 step 1 eq 1}
    	0 & \leq  H\bigl(\widetilde{u}_1, \{Y_i\}_{i=1}^n, \lambda\bigr) - H\bigl(u, \{Y_i\}_{i=1}^n, \lambda\bigr) \nonumber\\
 	   & \leq -\lambda + 2(1+\alpha)\alpha^{-1}\sigma^2 \log(n) + (1+\alpha)(1 - \beta)^2\frac{|I_1| |I_2|}{|I_1| +|I_2|} \left(\frac{\epsilon \omega_1}{\epsilon + \delta}\right)^2,
	\end{align}
	and
	\begin{align}\label{eq:one change point  case 3 step 1 eq 2}
    	& H\bigl(\widetilde{u}_2, \{Y_i\}_{i=1}^n, \lambda\bigr) - H\bigl(\widetilde{u}_1, \{Y_i\}_{i=1}^n, \lambda\bigr)    \nonumber \\
		\leq & \lambda - (1-\alpha)\frac{|I_1'| |I_2'|}{|I_1'| + |I_2'|} \left(\omega_1 - \frac{\omega_2 \delta_1}{\delta + \epsilon_1 + \delta_1}\right)^2 + \frac{2(1-\alpha)}{\alpha}\sigma^2\log(n) \nonumber \\
		\leq & \lambda + \frac{2(1-\alpha)}{\alpha}\sigma^2\log(n)  - (1-\alpha)(1 - \beta)^2 \frac{|I_1'| |I_2'|}{|I_1'| + |I_2'|} \omega_1^2,
	\end{align}
	where the last inequality of \eqref{eq:one change point  case 3 step 1 eq 2} follow from the observation that
	\begin{align*}
		 \frac{\omega_2 \delta_1}{\delta + \epsilon_1 + \delta_1} = \frac{\omega_2 \delta_1}{\epsilon_1 + \delta_1}\frac{\epsilon_1 + \delta_1}{\delta + \epsilon_1 + \delta_1} = \beta \omega_1 	\frac{\epsilon}{\epsilon + \delta}\frac{\epsilon_1 + \delta_1}{\delta + \epsilon_1 + \delta_1} \leq \beta\omega_1/2.
	\end{align*}

Equations~\eqref{eq:one change point  case 3 step 1 eq 1} and \eqref{eq:one change point  case 3 step 1 eq 2} lead to that
    \begin{align}
		0 & \leq H\bigl(\widetilde{u}_2, \{Y_i\}_{i=1}^n, \lambda\bigr) - H\bigl(u, \{Y_i\}_{i=1}^n, \lambda\bigr) \nonumber \\
		\leq & \frac{4}{\alpha}\sigma^2\log(n) + \frac{\omega_1^2 \epsilon (1 - \beta)^2\bigl\{(1+\alpha)\epsilon(\epsilon_1 + \delta_1) - (1 - \alpha)(\epsilon + \delta)(\delta + \epsilon_1 + \delta_1)\bigr\}}{(\epsilon + \delta + \epsilon_1 + \delta_1)(\epsilon + \delta)} \nonumber \\
		\leq & \frac{4}{\alpha}\sigma^2\log(n) - \frac{\omega_1^2(1 - \beta)^2 \epsilon}{4}, \label{eq-compare-u2-u}.
    \end{align}
  Plugging in \eqref{eq-compare-u2-u} into \eqref{eq:one change point  case 3 step 1 eq 1} with a choice of $\alpha = 1/4$ yields that
	\[
		\lambda \leq 50\sigma^2 \log(n),
	\]
	which is a contradiction.

\vskip 3mm
\noindent {\bf Step 2.}  Suppose $\omega_1\omega_2 < 0$.  Without loss of generality assume that with $\gamma \geq 1$ it holds that
	\[
		\left|\frac{\epsilon \omega_1}{\epsilon + \delta} \right| = \gamma \left|\frac{\delta_1\omega_2}{\epsilon_1 + \delta_1}\right|.
	\]

Since $\delta + \epsilon_1 = \eta_{k+1} - \eta_k + 1 > \Delta$, we have $\max\{\delta, \, \epsilon_1\} > \Delta/2$.

\vskip 1.5mm
\noindent {\bf case 1.} Suppose $\epsilon_1 > \Delta/2$.  It follows from \Cref{lemma:one side of one change point} that $\delta_1 < 8\lambda/\kappa_{k+1}^2$.  Then,
	\begin{align}\label{eq:one change point step 2 case 1}
		0 & \leq H\bigl(\widetilde{u}_1, \{Y_i\}_{i=1}^n, \lambda\bigr) - H\bigl(u, \{Y_i\}_{i=1}^n, \lambda\bigr) \nonumber \\
		& \leq  -\lambda + 2(1+\alpha)\alpha^{-1}\sigma^2 \log(n) + (1+\alpha)\frac{|I_1| |I_2|}{|I_1| +|I_2|} \left(\frac{\epsilon \omega_1}{\epsilon + \delta} - \frac{\delta_1\omega_2}{\epsilon_1 + \delta_1}\right)^2 \nonumber \\
		& \leq -\lambda +  2(1+\alpha)\alpha^{-1}\sigma^2 \log(n) + (\gamma + 1)^2(1+\alpha)\frac{|I_1| |I_2|}{|I_1| +|I_2|} \left(\frac{\delta_1\omega_2}{\epsilon_1 + \delta_1}\right)^2
	\end{align}
	and 
	\begin{align}\label{eq:one change point step 2 case 1-2}
		& H\bigl(\widetilde{u}_3, \{Y_i\}_{i=1}^n, \lambda\bigr) - H\bigl(\widetilde{u}_1, \{Y_i\}_{i=1}^n, \lambda\bigr) \nonumber \\
		& \leq \lambda - (1-\alpha)\frac{|I_1''| |I_2''|}{|I_1''| + |I_2''|} \left(\omega_2 - \frac{\omega_1 \epsilon}{\epsilon + \delta + \epsilon_1}\right)^2 + \frac{2(1-\alpha)}{\alpha}\sigma^2\log(n) \nonumber \\
		& \leq  \lambda - (1-\alpha)\frac{|I_1''| |I_2''|}{|I_1''| + |I_2''|} \left(1 - \gamma\frac{\delta_1}{\epsilon_1 + \delta_1}\frac{\epsilon + \delta}{\epsilon + \delta + \epsilon_1}\right)^2\omega_2^2+ \frac{2(1-\alpha)}{\alpha}\sigma^2\log(n).
	\end{align}

Equations~\eqref{eq:one change point step 2 case 1} and 	\eqref{eq:one change point step 2 case 1-2} lead to	
	\begin{align}\label{eq:one change point step 2 case 1-3}
		0 & \leq H\bigl(\widetilde{u}_3, \{Y_i\}_{i=1}^n, \lambda\bigr) - H\bigl(u, \{Y_i\}_{i=1}^n, \lambda\bigr) \nonumber \\
		& \leq \frac{4\sigma^2 \log(n)}{\alpha} + (\gamma + 1)^2(1+\alpha)\frac{|I_1| |I_2|}{|I_1| +|I_2|} \left(\frac{\delta_1\omega_2}{\epsilon_1 + \delta_1}\right)^2  \nonumber \\
		& \hspace{3cm} - (1-\alpha)\frac{|I_1''| |I_2''|}{|I_1''| + |I_2''|} \left(1 - \gamma\frac{\delta_1}{\epsilon_1 + \delta_1}\frac{\epsilon + \delta}{\epsilon + \delta + \epsilon_1}\right)^2\omega_2^2.
	\end{align}

Then there exists a sufficiently small $c > 0$ such that \eqref{eq:one change point step 2 case 1-3} yields
	\begin{align*} 
		c\omega_2^2 \delta_1 < \sigma^2\log(n),
	\end{align*}
	which can be plugged into \eqref{eq:one change point step 2 case 1} and shows that for a sufficiently large $C_1 > 0$, 
	\[
		\lambda \leq C_1\sigma^2\log(n).
	\]
	This contradicts the assumed condition on $\lambda$.

\vskip 1.5mm
\noindent {\bf case 2.} Suppose $\epsilon_1 \leq \Delta/2$.  It follows from \Cref{lemma:one side of one change point} that $\epsilon_1 < 8\lambda/\kappa_{k+1}^2$.  Then,
	\begin{align}
		& H\bigl(\widetilde{u}_2, \{Y_i\}_{i=1}^n, \lambda\bigr) - H\bigl(\widetilde{u}_1, \{Y_i\}_{i=1}^n, \lambda\bigr) \nonumber \\
		\leq & \lambda - (1-\alpha)\frac{|I_1'| |I_2'|}{|I_1'| + |I_2'|} \left(\omega_1 - \frac{\omega_2 \delta_1}{\delta + \epsilon_1 + \delta_1}\right)^2 + \frac{2(1-\alpha)}{\alpha}\sigma^2\log(n) \nonumber \\
		\leq & \lambda - (1- \alpha)\frac{|I_1'| |I_2'|}{|I_1'| + |I_2'|}\left(1 - \frac{1}{\gamma}\frac{\epsilon}{\epsilon + \delta} \frac{\epsilon_1 + \delta_1}{\delta + \epsilon_1 + \delta_1}\right)^2\omega_1^2 + \frac{2(1-\alpha)}{\alpha}\sigma^2\log(n). \label{eq-one-change-case2-1}
	\end{align}
	Equations~\eqref{eq:one change point step 2 case 1} and \eqref{eq-one-change-case2-1} lead to that there exists a sufficiently small $c > 0$ such that \eqref{eq:one change point step 2 case 1-3} yields
	\begin{align*} 
		c\omega_2^2 \epsilon_1 < \sigma^2\log(n),
	\end{align*}
	which can be plugged into \eqref{eq:one change point step 2 case 1} and shows that for a sufficiently large $C_1 > 0$, 
	\[
		\lambda \leq C_1\sigma^2\log(n).
	\]
	This again contradicts the assumed condition on $\lambda$.
\end{proof}

\subsection{Step 4: no changes}
\label{section:no change}

Suppose $I = (s_1, e] \in \widehat{\mathcal{P}}$ contains no true change point.  By symmetry, it suffices to show that there exists a large enough constant $C > 0$ such that
	\begin{equation}\label{eq-no-change-1}
		s_1 - \eta_k + 1 \leq C\lambda/\kappa_k^2.
	\end{equation}
	Assume $I_0 = (s_0, s_1] \in \widehat{\mathcal{P}}$.  We are to show the following.
	\begin{itemize}
		\item [(i)]	It is impossible that there is no true change point in $I_0 \cup I$.  This will be shown in \Cref{lemma:no change point consequtive}.
		\item [(ii)] If there exist exactly two true change points in $I_0$, then \eqref{eq-no-change-1} follows from \Cref{lemma:two change localization}.
		\item [(iii)] If there exists one and only one change point $\eta_k \in I_0$ and $s_1 - \eta_k < \eta_k - s_0$, then \eqref{eq-no-change-1} follows from \Cref{lemma:one side of one change point}.
		\item [(iv)] If there exists one and only one change point $\eta_k \in I_0$ and $s_1 - \eta_k \geq \eta_k - s_0$, it follows from \Cref{lemma:one change then no change} that this is impossible in the event of $\mathcal{B}$.
	\end{itemize}

\begin{lemma}\label{lemma:no change point consequtive}
Assume the inputs $\{Y_i\}_{i=1}^n$ satisfying Assumptions~\ref{assume:change} and \ref{assum-phase} and $\lambda = C_{\lambda}\sigma^2\log(n)$ with a sufficiently large $C_{\lambda} > 0$.  Assume that $I = (s_1, e] \in \widehat{\mathcal{P}}$ contains no change point.  Assume that $I_0 = (s_0, s_1] \in \widehat{\mathcal{P}}$.  Then in the event $\mathcal{B}$, there must exist a change point in $I_0$.
\end{lemma}

\begin{proof}
Let $J = I_0 \cup I$, $\widetilde{\mathcal{P}}$ be the interval partition such that 
	\[
		\widetilde{\mathcal{P}} = \widehat{\mathcal{P}} \cup \{J\} \setminus \{I_0, \, I\},
	\]
	and $\widetilde{u}$ be the piecewise-constant vector induced by $\widetilde{\mathcal{P}}$.  

Prove by contradiction, assuming that $J$ contains no change points.  Denote $\mu = \mathbb{E}(\overline Y_{I_0}) = \mathbb{E}(\overline Y_{I})$.  Then
	\begin{align*}
		0 & \leq H\bigl(\widetilde{u}, \{Y_i\}_{i=1}^n, \lambda\bigr) - H\bigl(\widehat{u}, \{Y_i\}_{i=1}^n, \lambda\bigr) = -\lambda + \frac{|I_0| |I|}{|I_0| + |I|} (\overline Y_{I_0} -\overline Y_{I})^2 \\
		& \leq -\lambda + \frac{|I_0| |I|}{|I_0| + |I|} (\overline Y_{I_0} - \mu - \overline Y_{I} + \mu)^2 \\
		& \leq -\lambda + \sigma^2 \log (n),
	\end{align*}
	where the last inequality follows from the definition of the event $\mathcal{B}$, and results in a contradiction with the condition on $\lambda$.
\end{proof}

\section{Proofs of the Results in  \Cref{section:bs}}
\label{sec:app-3}
%!TEX root = ./draft.tex

\subsection{Large probability events}

Define the events
	\begin{align}
		& \mathcal A_1 (\gamma) = \left\{ \sup_{0\leq s < t < e \leq n}\bigl|\widetilde{Y}^{s, e}_t - \widetilde{f}^{s, e}_t \bigr|\leq \gamma\right\}, \label{eq-event-A1} \\
		& \mathcal{A}_2(\gamma) =  \left\{ \sup_{0\le s < e \le n} \frac{\left|\sum_{i=s+1}^e (Y_i - f_i)\right|}{\sqrt{e-s}} \le \gamma\right\}, \label{eq-event-A2}
	\end{align}
	and
	\begin{equation}\label{eq-event-M}
	\mathcal{M} = \bigcap_{k = 1}^K \bigl\{s_m \in \mathcal{S}_k, e_m \in \mathcal{E}_k, \, \mbox{for some }m \in \{1, \ldots, M\}\bigr\}, 
	\end{equation}
	where $\{s_m\}_{m=1}^M$ and $\{e_m\}_{m=1}^M$ are two sequences independently selected at random in $(s, e)$, $\mathcal S_{k}= [\eta_k-3\Delta/4, \eta_k-\Delta/2 ]$ and $\mathcal E_{k}= [\eta_k+\Delta/2, \eta_k+3\Delta/4 ]$, $k = 1, \ldots, K$.

\begin{lemma}\label{lem:A1A2M}
For $\{Y_i\}_{i=1}^n$ satisfying \Cref{assume:change}, it holds that
	\begin{align*}
		& \mathbb{P}\bigl\{\mathcal{A}_1(\gamma)\bigr\} \geq 1 - e\cdot n^3 \exp\left(-c\gamma^2/\sigma^2\right), \\
		& \mathbb{P}\bigl\{\mathcal{A}_2(\gamma)\bigr\} \geq 1 - e\cdot n^2 \exp\left(-c\gamma^2/\sigma^2\right)
	\end{align*}
	and
	\[
		\mathbb{P}\bigl\{\mathcal{M}\bigr\} \geq 1 - \exp\left\{\log\left(\frac{n}{\Delta}\right) - \frac{M\Delta^2}{16n^2}\right\}.
	\]
\end{lemma}

\begin{proof}
Since for any suitable triples $(s, t, e)$, both $\bigl|\widetilde{Y}^{s, e}_t - \widetilde{f}^{s, e}_t\bigr|$ and $(e - s)^{-1/2}\left|\sum_{i=s+1}^e (Y_i - f_i)\right|$ can be written in the form $\left|\sum_{i=s+1}^e w_i X_i \right|$, where $X_i$'s are centred sub-Gaussian random variables and $w_i$'s satisfy $\sum_{i = s+1}^e w_i^2 = 1$.  
	
It follows from Hoeffding inequality that there exists an absolute constant $c > 0$ only depending on $\sigma$ such that
	\[
	\mathbb{P}\bigl\{\mathcal{A}^c_1(\gamma)\bigr\} \leq  e\cdot n^3 \exp\left(-c\gamma^2/\sigma^2\right) \mbox{ and } \mathbb{P}\bigl\{\mathcal{A}^c_2(\gamma)\bigr\} \leq  e\cdot n^2 \exp\left(-c\gamma^2/\sigma^2\right).
	\]

Since  the number of change points are bounded by $n/\Delta $, it holds that
	\begin{align*}
		\mathbb{P}\bigl\{\mathcal{M}^c\bigr\} & \leq \sum_{k=1}^K \prod_{m =1}^M \bigl\{1 - \mathbb{P}\bigl(s_m \in \mathcal{S}_k, e_m \in \mathcal{E}_k\bigr)\bigr\}\leq K \{1-\Delta^2/(16n^2)\}^M \leq n/\Delta (1 - \Delta^2/(16n^2))^M \\
		& \leq \exp\left\{\log\left(\frac{n}{\Delta}\right) - \frac{M\Delta^2}{16n^2}\right\}.
	\end{align*}
\end{proof}

\subsection{Technical Details for Step 1}

\begin{lemma}\label{lemma-step1-1}
Under \Cref{assume:change}, let $0 \leq s < \eta_k < e \leq n$ be any interval satisfying 
	\[
		\min\{\eta_k - s, e - \eta_k\} \ge c_1\Delta,
	\] 
	with $c_1 > 0$.  Then,
	\[
	\max_{s < t < e} \bigl|\widetilde{f}_{t}^{s,e}\bigr| \geq (c_1/2) \kappa \Delta (e-s)^{-1/2}.
	\]
\end{lemma}

\begin{proof}  
See Lemma~2.4 in \cite{venkatraman1992consistency}.
\end{proof}

\begin{lemma}\label{lemma:cusum two detected change point bound}
	Let $[s,e]$ contain only two  change points $\eta_{k}, \eta_{k+1}$.
		Then
		$$\sup_{s\le t\le e}  | \widetilde f^{s,e}_{t} |  \le  \sqrt {e-\eta_{k+1} }\kappa_{k+1}  + \sqrt{\eta_k-s} \kappa_k. $$
\end{lemma}

\begin{proof}
	Consider the sequence $\{g_t\}_{t=s+1}^e $ be such that 
		\[
		g_t=
		\begin{cases}
		f_{\eta_{k}}, & \text{if} \quad s +1\le  t< \eta_{k},
		\\
		f_t, & \text{if} \quad \eta_{k}  \le t \le  e. 
		\end{cases}
		\]

	For any  $t\ge \eta_k  $,
		\begin{align*}
		&\widetilde f^{s,e}_{t} - \widetilde g^{s,e}_{t} 
		\\
		=&
	\sqrt { \frac{e-t}{(e-s)(t-s)} }  \left(	\sum_{i=s+1}^{t} f_{i}  
	- \sum_{i=s+1}^{\eta_k} f _{\eta_k}  - \sum_{i=\eta_k +1}^{t  }f_i  
	 \right) \\
	- 	 &\sqrt { \frac{t-s}{(e-s)(e-t) } }  \left(	\sum_{i=t+1}^{e} f_{i} - \sum_{i=t+1}^{e} f_{i}
	 \right)
	 \\
	 =	&\sqrt { \frac{e-t}{(e-s)(t-s)} } (\eta_k-s) (f_{\eta_{k }} -f_{\eta_{k-1}})  . 
		\end{align*}
		So for $t \ge \eta_k$, $   |\widetilde f^{s,e}_{t} - \widetilde g^{s,e}_{t} |  \le \sqrt{\eta_k-s} \kappa_k$.
		Since $\sup_{s\le t\le e}    |\widetilde f^{s,e}_{t} |= \max\{  |\widetilde f^{s,e}_{\eta_k} | ,  |\widetilde f^{s,e}_{\eta_{k+1}}| \},
		$ and that 
		\begin{align*}
		 \max\{  |\widetilde f^{s,e}_{\eta_k}| , | \widetilde f^{s,e}_{\eta_{k+1}}|\}&
		\le \sup_{s\le t \le e} | \widetilde g^{s,e}_{t}   | +  \sqrt{\eta_k-s} \kappa_k
		\\
		&\le   \sqrt {e-\eta_{k+1} }\kappa_{k+1}  + \sqrt{\eta_r-s} \kappa_k
		\end{align*}
		where the last inequality follows form the fact that $g_t $ has only one change point in $[s,e]$.
			\end{proof}

\subsection{Technical details for Step 2}

In this section, eight results will be provided.  Before we go into details, we show the road map leading to complete the proof of \Cref{thm-wbs} in \Cref{fig-rm}.

\vskip 3mm
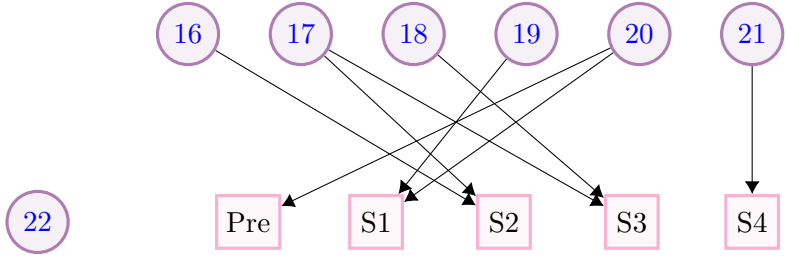
\begin{figure}
\centering
\begin{tikzpicture}[roundnode/.style={circle, draw=Fuchsia!60, fill=Fuchsia!5, very thick, minimum size=7mm}, squarednode/.style={rectangle, draw=CarnationPink!60, fill=CarnationPink!5, very thick, minimum size=7mm}]
%Nodes
\node[roundnode] (7) {\ref{lem-cov-7}};
\node[roundnode] (8) [right of= 7, xshift = 0.5cm] {\ref{lem-cov-8}};
\node[roundnode] (16) [right of= 8, xshift = 0.5cm] {\ref{lem-cov-16}};
\node[roundnode] (17) [right of= 16, xshift = 0.5cm] {\ref{lem-cov-17}};
\node[roundnode] (18) [right of= 17, xshift = 0.5cm] {\ref{lem-cov-18}};
\node[roundnode] (11) [right of= 18, xshift = 0.5cm] {\ref{lem-cov-11}};
\node[roundnode] (12) [below of= 7, yshift = -1.5cm, xshift = -2cm] {\ref{lem-cov-12}};

\node[squarednode] (pre) [below of= 8, yshift = -1.5cm, xshift = -0.7cm] {Pre};
\node[squarednode] (S1) [below of= 16, yshift = -1.5cm, xshift = -0.5cm] {S1};
\node[squarednode] (S2) [below of= 17, yshift = -1.5cm, xshift = -0.3cm] {S2};
\node[squarednode] (S3) [below of= 18, yshift = -1.5cm, xshift = -0.1cm] {S3};
\node[squarednode] (S4) [below of= 11, yshift = -1.5cm, xshift = 0cm] {S4};
 
%Lines
\draw[-{Latex[width=2mm]}] (18) -- (pre);
\draw[-{Latex[width=2mm]}] (17) -- (S1);
\draw[-{Latex[width=2mm]}] (18) -- (S1);
\draw[-{Latex[width=2mm]}] (7) -- (S2);
\draw[-{Latex[width=2mm]}] (8) -- (S2);
\draw[-{Latex[width=2mm]}] (16) -- (S3);
\draw[-{Latex[width=2mm]}] (8) -- (S3);
\draw[-{Latex[width=2mm]}] (11) -- (S4);

\end{tikzpicture}
\caption{Road map to complete the Step 2 in the proof of \Cref{thm-wbs}.  The circles are lemmas, and the squares are the steps in the proof of \Cref{lem-cov-12}.  The directed edges mean the heads of the edges are used in the tails of the edges. \label{fig-rm}}
\end{figure}

\vskip 3mm

\begin{lemma}\label{lem-cov-7}
	Suppose $(s, e) \subset (0, n)$ is a generic interval satisfying
		\begin{equation*}
			\eta_{k-1} \le s \le \eta_k \le \ldots \le \eta_{k+q} \le e \le \eta_{k+q+1}, \quad q\ge 0.
		\end{equation*} 
	Then there exists a continuous function $\widetilde F_{t}^{s,e}: \, [s, e]\to \mathbb R$ such that $\widetilde F_{t}^{s,e} = \widetilde f_{t}^{s, e}$ for every $t \in [s, e] \cap \mathbb Z$ with the following additional properties.
	\begin{itemize}
	\item [(i)]	$|\widetilde F_t^{s,e}|$ is maximized at the change points within $[s, e]$.  In other words,
		\[
			\argmax_{s\le t\le e}  |\widetilde F_t^{s,e}| \cap \bigl\{\eta_{k},\ldots, \eta_{k+q}\bigr\} \neq \emptyset.
		\]
	\item [(ii)] If $\widetilde F_t^{s,e} > 0$ for some $ t \in (s, e)$, then $\widetilde F_t^{s,e}$ is either monotonic or decreases and then increases within each of the interval $(s, \eta_k), \ldots, (\eta_{k+q}, e)$. 
	\end{itemize}
\end{lemma}

The proof of \Cref{lem-cov-7} can be found in Lemmas 2.2 and 2.3 of  \cite{venkatraman1992consistency}.  We remark that if $\widetilde F_t^{s,e}\le 0 $ for all $t \in (s, e)$, then it suffices to consider the time series $\{-f_i\}_{i=1}^n$ and a similar result as in the second part of   \Cref{lem-cov-7} still holds.

Our next lemma is an adaptation of a result first obtained by \cite{venkatraman1992consistency}, which quantifies how fast the CUSUM statistics decays around a good change point.

\begin{lemma}[\cite{venkatraman1992consistency} Lemma 2.6]\label{lem-cov-8}
Let $[s, e]\subset [1,n]$ be any generic interval.  For some $c_1, c_2 >0$ and $\gamma > 0$ such that
	\begin{align}
		& \min \{\eta_k-s , e-\eta_k\} \ge c_1\Delta, \label{eq:ven 0} \\
		& \widetilde{f}_{\eta_k} \ge c_2\kappa  \Delta (e-s)^{-1/2}, \label{eq:ven 1}
	\end{align}
	and suppose there exists a sufficiently small constant $c_3 > 0$ such that 
	\begin{align}
		\max_{s\le t\le e} |\tf_t | - \tf_{\eta_k} \le 2\gamma  \le  c_3 \kappa \Delta^3 (e-s)^{-5/2}. \label{eq:ven 2}
	\end{align}
	Then there exists an absolute constant $c>0$ such that if the point $d\in [s,e]$ is such that $|d-\eta_k| \le  c_1\Delta/16$, then
	\[
		\widetilde f_{\eta_k}^{s,e} - \widetilde f_{d}^{s,e}   > c \widetilde f^{s,e}_{\eta_k} |\eta_k-d | \Delta(e-s)^{-2}. 
	\]
\end{lemma}

\begin{proof} 
Without loss of generality, assume that $d \ge \eta_k$.  Following the argument of \cite{venkatraman1992consistency} Lemma~2.6, it suffices to consider two cases: (1) $\eta_{k+1} > e$, and (2) $\eta_{k+1} \le e$.

\vskip 3mm
\noindent \textbf{Case 1.} Let $E_l$ be defined as in the case 1 in \cite{venkatraman1992consistency} Lemma~2.6.  There exists a $c' > 0$ such that, for every $d\in [\eta_{k},\eta_k+ c_1\Delta/16]$, $\widetilde f_{\eta_k}^{s,e} - \widetilde f_{d}^{s,e}$ (which in the notation of \cite{venkatraman1992consistency} is the term $E_l$) can be written as 
	\begin{align*}
		& \tf_{\eta_k} | d-\eta_k|\frac{e-s }{\sqrt {e-\eta_k} \sqrt {\eta_k-s  +(d-\eta_k)} } \\
		& \hspace{3cm}\times  \frac{1}{\sqrt { (\eta_k-s  +(d-\eta_k) ) (e-\eta_k)} + \sqrt { (\eta_k -s ) (e-\eta_k -(d-\eta_k))}}.
    \end{align*}
	Using the inequality $(e-s)\ge 2c_1\Delta $, the previous expression is lower bounded by
	\[
		c'|d-\eta_k| \tf_{\eta_k}\Delta (e-s)^{-2}.
    \]

\vskip 3mm
\noindent  \textbf{Case 2.} Let $h = c_1\Delta/8 $  and  $l=d-\eta_k \le h/2$. Then, following closely the initial calculations for case 2 of Lemma 2.6 of \cite{venkatraman1992consistency}, we obtain that
	\[
		\widetilde f_{\eta_k}^{s,e}   -\widetilde f_{d}^{s,e}  \ge E_{1l} (1+E_{2l}) +E_{3l},
	\]
	where
	\begin{align*}
		E_{1l} &=\frac{ \tf_{\eta_k} l(h-l) } {\sqrt{(\eta_k-s+l )(e-\eta_k -l)   }   \left( \sqrt{ (\eta_k-s+l )(e-\eta_k -l)} +\sqrt{ (\eta_k-s)(e-\eta_k)} \right)}, \\
        E_{2l} &= \frac{((e-\eta_k-h )- (\eta_k-s))((e-\eta_k-h )- (\eta_k-s) -l)}{\sqrt{ (\eta_k-s+l )(e-\eta_k -l)} + \sqrt{ (\eta_k-s+h)(e-\eta_k-h)}} \\
        & \hspace{3cm} \times \frac{1}{\sqrt{ (\eta_k- s)(e-\eta_k)} +\sqrt{ (\eta_k-s+h ) (e-\eta_k-h)}},\\
	\end{align*}
	and
	\begin{align*}
		E_{3l} = -\frac{(\tf_{\eta_k+h}  -\tf_{\eta_k} )l }{h}  \sqrt{\frac{ (\eta_k -s+h)(e-\eta_k-h) }{  (\eta_k-s+l )(e-\eta_k-l) }}.
	\end{align*}
	Since  $h= c_1\Delta/8$ and $l\le h/2$, it holds that
	\begin{align*}
		E_{1l} \ge (c_1/16)\tf_{\eta_k} | d-\eta| \Delta (e-s)^{-2}.
	\end{align*}
	Observe that 
	\begin{equation}\label{eq:l and h-1}
		\eta_k-s\le  \eta_k-s +l \le \eta_k-s +h \le 9(\eta_k-s)/8
	\end{equation}
	and
	\begin{equation}\label{eq:l and h-2}
		 e-\eta_k \ge e-\eta_k-l \ge e-\eta_k -h \ge 7(e-\eta_k)/8. 
	\end{equation}
	Thus
	\begin{align*}
		E_{2l } & = \frac{((e-\eta_k-h )- (\eta_k-s))^2 + l(h+ \eta_k-s) -l (e-\eta_k) }{\left( \sqrt{ (\eta_k-s+l )(e-\eta_k -l)} + \sqrt{ (\eta_k-s+h)(e-\eta_k-h)} \right)} \\
		& \hspace{3cm} \times  \frac{1}{\left(\sqrt{ (\eta_k- s)(e-\eta_k)} +\sqrt{ (\eta_k-s+h ) (e-\eta_k-h)} \right) } \\
		\ge & \frac{- l(e-\eta_k) }{ (\eta_k-s+h)(e-\eta_k-h) } \ge  \frac{- l(e-\eta_k) }{ (\eta_k-s)(7/8)(e-\eta_k) } \ge -1/2,
	\end{align*}
	where \eqref{eq:l and h-1} and \eqref{eq:l and h-2} are used in the second inequality and the fact that $l\le h/2 \le c_1\Delta/16 \le (\eta_k-s )/16 $ is used in the last inequality.
	
For $E_{3l}$, observe that 
	\[
		\tf_{\eta_k+h}  -\tf_{\eta_k} \le  | \tf_{\eta_k+h} | -\tf_{\eta_k} \le \max_{s\le t\le e} |   \tf_{t}  |-\tf_{\eta_k} \le 2\gamma.
	\]
	This combines with \eqref{eq:ven 0} and that $l/2 \le h= c_1\Delta/8$, implying that 
	\[
		\eta_k-s\le  \eta_k-s +l \le \eta_k-s +h \le 9(\eta_k-s)/8 \mbox{ and } e-\eta_k \ge e-\eta_k-l \ge e-\eta_k -h \ge 7(e-\eta_k)/8.
	\]
	Therefore, with a sufficiently small constant $c'' > 0$, it holds that
	\begin{align*}
		E_{3l} & \ge - \frac{ 2(d-\eta_k) \gamma}{c_1\Delta/8}  \sqrt { \frac{(9/8)(\eta_k-s) (e- \eta_k) }{ (\eta_k-s)(7/8) (e-\eta_k )} }  \ge - \frac{ 32(d-\eta_k) \gamma}{c_1\Delta}  \\
		& \ge -(c''/4) \tf_{\eta_k}   (d-\eta_k) \Delta (e-s)^{-2},
	\end{align*}
	where the first inequality follows from \eqref{eq:l and h-1} and \eqref{eq:l and h-1}, and the last inequality follows from \eqref{eq:ven 1} and \eqref{eq:ven 2}.  Thus,
	\[
		\widetilde f_{\eta_k}^{s,e}   -\widetilde f_{d}^{s,e}  \ge E_{1l} (1+E_{2l}) +E_{3l} \ge (c''/4) \tf_{\eta_k}   | \eta_k-d| \Delta (e-s)^{-2}.
	\]
\end{proof}

\begin{lemma}\label{lem-cov-16}
Suppose $[s, e] \subset [1, n]$ such that $e-s\le C_R\Delta$, and that 
	\[
		\eta_{k-1} \le s\le \eta_k \le \ldots\le \eta_{k+q} \le e \le \eta_{k+q+1}, \quad q\ge 0.
	\]
	Denote
	\[
		\kappa_{\max}^{s,e} =\max \{\eta_{p} - \eta_{p-1}:\, k\le p \le k+q\}.
	\]
	Then for any $k-1 \le p \le k+q$, it holds that 
	\[
		\left|\frac{1}{e-s}\sum_{i=s}^e f_i - f_{\eta_p} \right| \le C_R\kse.
	\]
\end{lemma}

\begin{proof}
Since $e-s\le C_R\Delta$, the interval $[s,e]$ contains at most $C_R+1$ change points.  Observe that 
	\begin{align*}
 		& \left|\frac{1}{e-s}\sum_{i=s}^e f_i - f_{\eta_p} \right| \\
		= & \frac{1}{e-s} \left|  \sum_{i=s}^{\eta_k} (f_{\eta_{k-1}} - f_{\eta_p}) +  \sum_{i={\eta_k +1}}^{\eta_{k+1}}(f_{\eta_{k}} - f_{\eta_p}) + \ldots + \sum_{i={\eta_{k+q}+1}}^{e} (f_{\eta_{k+q}} - f_{\eta_p}) \right| \\
		\le & \frac{1}{e-s}  \sum_{i=s}^{\eta_k} |p-k|\kse +  \sum_{i={\eta_k +1}}^{\eta_{k+1}} |p-k-1|\kse  + \ldots + \sum_{i={\eta_{k+q}+1}}^{e} |p-k-q-1|\kse \\
		\le & \frac{1}{e-s} \sum_{i=s}^{e}  (C_R+1)\kse,
	\end{align*}
	where $|p_1 -p_2| \le C_R +1 $ for any $\eta_{p_1}, \eta_{p_2} \in [s,e]$ is used in the last inequality.
\end{proof}

\begin{lemma}\label{lem-cov-17}
If $\eta_k$ is the only change point in $(s, e)$, then
	\[
		|\widetilde f^{s,e}_{\eta_k} | = \sqrt { \frac{(\eta_k-s)(e-\eta_k)}{e-s}  } \kappa_k \le \sqrt { \min \{\eta_k -s , e -\eta_k\} } \kappa_k.
	\]
\end{lemma}

\begin{lemma}\label{lem-cov-18}
Let $(s, e) \subset (0, n)$ contains two or more change points such that 
	\[
		\eta_{k-1} \le s\le \eta_k \le \ldots\le \eta_{k+q} \le e \le \eta_{k+q+1}, \quad q\ge 1.
	\]
	If $\eta_{k}-s \le  c_1\Delta$, for $c_1 > 0$, then
	\[
		|\widetilde f^{s,e}_{\eta_k}| \le \sqrt{c_1} |  \widetilde f^{s,e}_{\eta_{k+1}}| +2\kappa_k  \sqrt {\eta_k -s}. 
	\]	
\end{lemma}

\begin{proof}
Consider the sequence $\{g_t\}_{t=s+1}^e $ be such that 
	\[
		g_t = \begin{cases}
			f_{\eta_{r+1}}, & s +1\le  t\le \eta_{k}, \\
			f_t, & \eta_{k}+1 \le t \le e. 
		\end{cases}
	\]
	For any  $t\ge \eta_r$, it holds that
	\begin{align*}
		\widetilde f^{s,e}_{\eta_k} - \widetilde g^{s,e}_{\eta_k} = \sqrt { \frac{(e-s)-t}{(e-s)(t-s)} } (\eta_k-s) (f_{\eta_{k+1}} -f_{\eta_{k}})  \le \sqrt{\eta_k-s} \kappa_k.
	\end{align*}
	Thus, 
	\begin{align*}
		|\widetilde f^{s,e}_{\eta_k} | & \le |\widetilde g^{s,e}_{\eta_k}  |+  \sqrt{\eta_k-s} \kappa_k \le \sqrt { \frac{(\eta_k-s)  (e-\eta_{k+1})  }{   ( \eta_{k+1}-s)  (e-\eta_k)    } }|\widetilde g^{s,e}_{\eta_{k+1}} |+  \sqrt{\eta_k-s} \kappa_k \\
		& \le \sqrt { \frac{c_1\Delta }{ \Delta}}|\widetilde g^{s,e}_{\eta_{k+1}} |  +  \sqrt{\eta_k-s} \kappa_k \le \sqrt{c_1} |\widetilde f^{s,e}_{\eta_{k+1}} |  + 2\sqrt{\eta_k-s} \kappa_k,
	\end{align*}
	where the first inequality follows from the observation that the first change point of $g_t$ in $(s, e)$ is at $\eta_{k+1}$.	
\end{proof}

For a pair $(s,e)$ of positive integers with $s < e$, let $\mathcal{W}_d^{s,e}$ be the two dimensional linear subspace of $\mathbb{R}^{(e-s)}$ spanned by the vectors 
	\[
		u_1 = (\underbrace{1, \ldots, 1}_{d-s}, \underbrace{0, \ldots, 0}_{e-d})^{\top} \mbox{ and } u_2 = (\underbrace{0, \ldots, 0}_{d-s}, \underbrace{1, \ldots, 1}_{e-d})^{\top}.
	\]
	For clarity, in the lemma below, we will use $\langle \cdot ,\, \cdot \rangle $ to denote the inner product of two vectors in the Euclidean space.

\begin{lemma}\label{lem-cov-11}
For $x = (x_{s+1}, \ldots, x_e)^{\top} \in \mathbb{R}^{(e-s)}$, let $\p^{s,e}_d(x)$ be the projection of $x$ onto $\mathcal{W}^{s,e}_d$.
	\begin{enumerate}
		\item[(i)] The projection $\mathcal{P}^{s,e}_d(x)$ satisfies
		\[
			\p^{s,e}_d (x) = \frac{1}{e-s}\sum_{i=s+1}^e x_i+\langle x,\psi^{s,e}_d\rangle \psi^{s,e}_d, 
		\]
		where $\langle \cdot, \cdot \rangle$ is the inner product in Euclidean space, and $\psi^{s,e}_d = ( (\psi^{s,e}_d)_s, \ldots, (\psi^{s,e}_d)_{e-s})^{\top}$ with
		\[
			(\psi^{s,e}_d)_i=  \begin{cases}
				\sqrt\frac{e-d}{(e-s)(d-s)}, &  i = s+1, \ldots, d, \\
				-\sqrt\frac{d-s}{(e-s)(e-d)}, &  i = d+1, \ldots, e,
			\end{cases}
		\]
		i.e. the $i$-th entry of $\p^{s,e}_d (x)$ satisfies 
		\[
			\p^{s,e}_d (x)_i= \begin{cases}
				\frac{1}{d-s}\sum_{j=s+1}^d x_j, &  i = s+1, \ldots, d, \\
				\frac{1}{e-d}\sum_{j=d+1}^e x_j, & i = d+1, \ldots, e.
			\end{cases}  
		\]
	
	\item [(ii)] Let $\bar x =\frac{1}{e-s}\sum_{i=s+1}^e x_i$. Since $\langle \bar x,\psi^{s,e}_d \rangle=0$, it holds that
		\begin{equation} \label{eq:anova}
			\| x - \p^{s,e}_d(x) \|^2 = \| x-\bar x \|^2 - \langle x, \psi^{s,e}_d\rangle ^2.
		\end{equation}
	\end{enumerate}	
\end{lemma}

\begin{proof}
The results hold following the fact that the projection matrix of subspace $\mathcal{W}^{s,e}_d$ is
	\[
		P^{s,e}_{\mathcal{W}^{s,e}_d} = \left(\begin{array}{cccccc}
			1/(d-s) & \cdots & 1/(d-s)& 0 & \cdots & 0 \\
			\vdots & \vdots & \vdots & \vdots & \vdots & \vdots \\
			1/(d-s) & \cdots & 1/(d-s) & 0 & \cdots & 0 \\
			0 & \cdots & 0 & 1/(e-d) & \cdots & 1/(e-d) \\
			\vdots & \vdots & \vdots & \vdots & \vdots & \vdots \\
			0 & \cdots & 0 & 1/(e-d)& \cdots & 1/(e-d)
		\end{array}\right).
	\]
\end{proof}

\begin{lemma}\label{lem-cov-12}
Under \Cref{assume:change}, let $(s_0, e_0)$ be an interval with $e_0 - s_0 \le C_R\Delta$ and contain at lest one change point $\eta_k$ such that 
	\[
		\eta_{k-1} \le s_0 \le \eta_k \le \ldots \le \eta_{k+q} \le e_0 \le \eta_{k+q+1}, \quad q\ge 0.
	\]
	Suppose that there exists $k'$ such that $\min\{\eta_{k'} - s_0,\, e_0 -\eta_{k'} \}\ge \Delta/16$.  Let $\kse= \max\{\kappa_p: \min\{ \eta_p -s_0 , e_0 -\eta_p \} \ge \Delta /16\}$.   Consider any generic $[s, e] \subset [s_0,e_0]$, satisfying
	\[
		\min\{ \eta_{k} -s_0 , e_0 -\eta_{k} \} \ge \Delta /16 \quad \text{for all } \eta_k \in[s,e].
	\]

Let $b \in \arg \max_{s < t < e}|\widetilde Y_{t}^{s,e}|$.  For some $c_1>0$ and $\gamma>0$ , suppose that
	\begin{align}
		& |\widetilde Y_{b}^{s,e}  |  \ge c_1 \kse \sqrt{\Delta},  \label{eq:wbs size of sample} \\
		& \sup_{s < t < e} |\widetilde Y_{t}^{s,e}   - \widetilde f_{t}^{s,e} | \le \gamma, \label{eq-wbs-lambda}
	\end{align}
	and
	\begin{equation}\label{eq:wbs noise 2}
		\sup_{s_1 < t < e_1} \frac{1}{\sqrt{e_1-s_1}}\left| \sum_{t=s_1+1}^{e_1} (Y_t - f_t )\right| \le  \gamma. 
	\end{equation}
	If there exists a sufficiently small $0 < c_2 < c_1/2$ such that
	\begin{equation}\label{eq:wbs noise}
		\gamma\le c_2\kse\sqrt \Delta,
	\end{equation}
	then there exists a change point $\eta_{k} \in (s, e)$  such that 
	\[
    	\min \{e-\eta_k,\eta_k-s\}  >  \Delta /4  \quad \text{and} \quad |\eta_{k} -b |\le C_3\gamma^2\kappa_k^{-2}.
	\]	
\end{lemma}

\begin{proof}

Without loss of generality, assume that $\widetilde f_{b}^{s,e}>0$ and that $\widetilde f_t^{s,e} $ is locally decreasing at $b$.  Observe that there has to be a change point $\eta_k \in [s,b]$, or otherwise $\widetilde f_b^{s,e} >0 $  implies that  $\widetilde f_t^{s,e} $ is decreasing, as a consequence of  \Cref{lem-cov-18}.  Thus, if $s\le \eta_k\le b \le  e $, then 
	\begin{align}
		\widetilde f_{\eta_k}^{s,e}\ge \widetilde f_{b}^{s,e} \ge |\widetilde Y^{s,e}_b  | -\gamma  \ge c_1 \kse \sqrt{\Delta} -c_2 \kse \sqrt \Delta  \ge  (c_1/2) \kse \sqrt {\Delta}. \label{eq:wbs size of change point}
	\end{align}
	Observe that $e-s\le e_0-s_0\le C_R\Delta $ and that $(s, e)$ has to contain at least one change point or otherwise $ |\widetilde f^{s,e}_{\eta_k} |  =0 $ which contradicts  \eqref{eq:wbs size of change point}.
	
We decompose the rest of the proof in four steps.  {\bf Step 1} shows that $\eta_k$ is far enough away from end points $s$ and $e$.  {\bf Step 2} utilizes \Cref{lem-cov-8} -- the machinery originally developed for BS in \cite{venkatraman1992consistency} -- to show that $b$ is not far away from $\eta_k$.  This is actually a consistent estimator, but not optimal.  {\bf Step 3}	brings in the WBS techniques to refine the error bound, which is {\it de facto} optimal.  The proof is completed in {\bf Step 4}.

\vskip 3mm
\noindent {\bf Step 1.} 
In this step, we are to show that $\min\{ \eta_k -s , e -\eta_k \} \ge \min\{1, c_1^2 \}\Delta /16$. 

Suppose $\eta_k$ is the only change point in $(s, e)$.  So $\min\{ \eta_k -s , e -\eta_k \} \ge   \min\{1, c_1^2\}\Delta /16$ must hold or otherwise it follows from \Cref{lem-cov-17}, we have
	\[
	    |\widetilde f^{s,e}_{\eta_k} | < \frac{c_1 }{4} \kappa_k
	    \sqrt{\Delta} \le \frac{c_1}{2}  \kse \sqrt{\Delta},
	\]
	which contradicts \eqref{eq:wbs size of change point}.

Suppose $(s, e)$ contains at least two change points. Then $ \eta_k -s \le   \min\{1,c_1^2  \}\Delta /16 $ implies that $\eta_k$ is the first change point in $[s,e]$.   Therefore,
	\begin{align*}
		|\widetilde f^{s,e}_{\eta_k}| \le \frac{1}{4}  |  \widetilde f^{s,e}_{\eta_{k+1}}| +2\kappa_k  \sqrt {\eta_k -s}   \le \frac{1}{4}\max_{s < t < e}|  \widetilde f^{s,e}_{t}|+\frac{c_1}{2}\kappa_k \sqrt {\Delta}  \\
		\le \frac{1}{4}|  \widetilde Y^{s,e}_{b}| +\gamma +\frac{c_1}{2}\kse  \sqrt {\Delta}  \le\frac{3}{4}|  \widetilde Y^{s,e}_{b}| +\gamma< |  \widetilde Y^{s,e}_{b}|  -\gamma,
	\end{align*}
	where the first inequality follows from \Cref{lem-cov-18}, the fourth inequality follows from \eqref{eq:wbs size of sample}, and the last inequality holds when $c_2$ is sufficiently small.  This contradicts with \eqref{eq:wbs size of change point}.

\vskip 3mm
\noindent{\bf Step 2.} By \Cref{lem-cov-8} there exists $d$ such that $d\in[\eta_{k},\eta_{k} + \gamma \sqrt \Delta (\kse)^{-1}]$ and that $\widetilde f_{\eta_{k}}^{s,e} -\widetilde f_{d}^{s,e}> 2\gamma$.  For the sake of contradiction, suppose $b \ge d$.  Then 
	\[
		\widetilde f_{b}^{s,e} \le  \widetilde f_{d}^{s,e}  <  \widetilde f_{\eta_{k}}^{s,e} -2\gamma  \le \max_{s < t < e}|\widetilde f_{t}^{s,e} | -2\gamma \le \max_{s < t < e} |\widetilde Y^{s,e}_t| +\gamma-2 \gamma  = |\widetilde Y^{s,e}_b| -\gamma,
	\]
	where the first inequality follows from \Cref{lem-cov-7}, which ensures that  $\widetilde f_{t}^{s,e}$ is decreasing on $[\eta_{k},b]$ and $d \in [\eta_{k}, b]$.  This is a contradiction to \eqref{eq:wbs size of change point}.  Thus $ b\in [\eta_{k},\eta_{k} + \gamma \sqrt \Delta (\kse)^{-1} ]$. 

\vskip 3mm
\noindent {\bf Step 3.}  Let $f^{s,e} =(f_{s+1},\ldots, f_e)^{\top} \in \mathbb{R}^{(e-s)} $ and $Y^{s,e}=(Y_{s+1}, \ldots, Y_e)^{\top} \in \mathbb{R}^{(e-s)}$.  By the definition of $b$, it holds that
	\[
		\bigl\|Y^{s,e} - \mathcal{P}^{s,e}_{b}(Y^{s,e})\bigr\|^2 \leq \bigl \|Y^{s,e} - \mathcal{P}^{s,e}_{\eta_k}(Y^{s,e})\bigr\|^2  \leq \bigl\|Y^{s,e} - \mathcal{P}_{\eta_k}^{s,e}(f^{s,e})\bigr\|^2.
	\]
	For the sake of contradiction, throughout the rest of this argument suppose that, for some sufficiently large constant $C_3 > 0$ to be specified,
	\begin{align}\label{eq:wbs contradict assume}
		\eta_k + \max\{C_3\gamma^2\kappa_k^{-2},\delta\}< b.
	\end{align}
	(This will of course imply that $\eta_k + \max\{C_3\gamma^2 (\kse)^{-2},\delta\}< b$).  We will show that this leads to the bound
	\begin{align}\label{eq:WBS sufficient}
		\bigl\|Y^{s,e} - \mathcal{P}_{b}^{s,e} (Y^{s,e})\bigr\|^2 > \bigl\|Y^{s,e} - \mathcal{P}^{s,e}_{\eta_k}(f^{s,e})\bigr\|^2,	
    \end{align}
	which is a contradiction. 

To derive \eqref{eq:WBS sufficient} from \eqref{eq:wbs contradict assume}, we note that  $\min\{ e-\eta_k,\eta_k-s\}\ge  \min\{1, c_1^2 \}\Delta/16$ and that $| b- \eta_k| \le  \gamma \sqrt \Delta (\kse)^{-1}$ implies that 
	\begin{align}\label{eq:wbs size of intervals}
		\min\{ e-b, b-s\} \ge   \min\{1, c_1^2 \}\Delta /16 -\gamma \sqrt \Delta (\kse)^{-1} \ge \min\{1, c_1^2 \}\Delta /32 ,
	\end{align}
	where the last inequality follows from \eqref{eq:wbs noise} and holds for an appropriately small $c_2>0$.

\Cref{eq:WBS sufficient} is in turn implied by 
	\begin{equation} \label{eq:WBS sufficient 2}
		2\langle \varepsilon^{s,e} ,\p_b(Y^{s,e}) - \p _{\eta_k}(f^{(s,e)})\rangle < \|f^{s,e}-\p_b(f^{s,e}) \|^2 -\|f^{s,e}-\p_{\eta_k}(f^{s,e}) \|^2,
	\end{equation}
	where $\varepsilon^{s,e}= Y^{s,e}-f^{s,e}$.  By \eqref{eq:anova}, the right hand side of \eqref{eq:WBS sufficient 2} satisfies the relationship with sufficiently small absolute constants $c, c' > 0$,
	\begin{align*}
		& \|f^{s,e}-\p_b(f^{s,e}) \|^2 -\|f^{s,e}-\p_{\eta_k}(f^{s,e}) \|^2 = \langle f^{s,e} , \psi_{\eta_k} \rangle^2 -\langle f^{s,e} , \psi_{b} \rangle^2  \\
		= &  (\tf_{\eta_k})^2 -(\tf_{b})^2  \ge ( \tf_{\eta_k} - \tf_{b} ) | \tf_{\eta_k}| \ge  c |d-\eta_k|  (\tf_{\eta_k})^2 \Delta^{-1}\\
		\ge &  c' |d-\eta_k |(\kse)^2,
	\end{align*}
	where \Cref{lem-cov-8} and \eqref{eq:wbs size of change point} are used in the second and third inequalities.  The left hand side of \eqref{eq:WBS sufficient 2} can  in turn be rewritten as  
	\begin{equation} \label{eq:perturbations}
		2\langle \varepsilon^{s,e} ,\p_b(X^{s,e}) - \p _{\eta_k}(f^{s,e})\rangle = 2 \langle \varepsilon^{s,e}, \p_b(X^{s,e})-\p_b(f^{s,e})\rangle + 2  \langle \varepsilon^{s,e}, \p_b(f^{s,e})-\p_{\eta_k} (f^{s,e})\rangle .
	\end{equation}

The second term on the right hand side  of the previous display can be decomposed as
	\begin{align*}
		\langle \varepsilon^{s,e}  , \p_b (f ^{s,e})-\p_{\eta_k} (f^{s,e} )\rangle &  =  \left( \sum_{i=s+1}^{\eta_k} +\sum_{i={\eta_k}+1}^b +\sum_{i=b+1}^e\right) \varepsilon^{s,e}_i \left( \p_b(f^{s,e})_i  - \p_{\eta_k}(f^{s,e})_i  \right)\\
 		&= I +II +III.
	\end{align*}
	In order to bound the terms $I$, $II$ and $III$, observe that, since $e-s\le e_0-s_0\le  C_R\Delta$, the interval $[s,e]$ must contain at most $C_R+1$ change points. % Let 
	%\[
	%	\eta_{r'-1}< s\le \eta_{r'} \le \ldots  \le \eta_{p'+q'}< e\le \eta_{p'+q'+1}. 
	%\]
	% Then $p'+q'+1-r'\le C_R +1$. 

\vskip 3mm
\noindent {\bf Step 3.1.} We can write 
	\[
	I = \sqrt{{\eta_k} -s}\left (\frac{1}{\sqrt{{\eta_k} -s}} \sum_{i=s+1}^{\eta_k} \varepsilon^{s,e}_i\right) \left( \frac{1}{b-s} \sum_{i=s+1}^b f_i-\frac{1}{{\eta_k}-s} \sum_{i=s+1}^{\eta_k} f_i\right). 
	\]
	Thus,
	\begin{align*} 
		& \left |\frac{1}{b-s} \sum_{i=s+1}^b f_i  -\frac{1}{{\eta_k}-s} \sum_{i=s+1}^{\eta_k} f_i \right | = \left|  \frac{ (\eta_k -s ) (\sum_{i=s+1}^{\eta_k}  f_i +\sum_{i=\eta_k+1}^{b}  f_i)  - (b-s) \sum_{i=s+1}^{\eta_k} f_i }{(b-s)(\eta_k -s)}   \right|\\
		=& \left|  \frac{ (\eta_k -b ) \sum_{i=s+1}^{\eta_k}  f_i + (\eta_k-s)\sum_{i=\eta_k+1}^{b}  f_i)  }{(b-s)(\eta_k -s)}   \right| =   \left|  \frac{ (\eta_k -b ) \sum_{i=s+1}^{\eta_k}  f_i + (\eta_k-s) (b-\eta_k) f_{\eta_k+1})  }{(b-s)(\eta_k -s)}   \right|\\
		=& \frac{b-\eta_k}{b-s}\left|  - \frac{1}{\eta_k -s } \sum_{i=s+1}^{\eta_k}  f_i+ f_{\eta_{k+1}}  \right| \le \frac{b-\eta_k}{b-s}  (C_R +1 )\kse
	\end{align*}
	where \Cref{lem-cov-16} is used in the last inequality.  It follows from \Cref{eq:wbs noise 2} that
	\begin{align*}
		| I|\le  \sqrt{\eta_k-s}\gamma  \frac{|b-\eta_k|}{b-s}(C_R+1) \kse \le \frac{4\sqrt{2}}{\min\{1, c_1\}}|b-\eta_k | \Delta^{-1/2} \gamma  (C_R+1) \kse,
	\end{align*}
	where \eqref{eq:wbs size of intervals} is used in the last inequality.

\vskip 3mm
\noindent {\bf Step 3.2.}  For the second term $II$, we have that 
	\begin{align*} 
		|II| =&\left| \sqrt{{b-\eta_k} }\left (\frac{1}{\sqrt{{b-\eta_k}}} \sum_{i={\eta_k+1}}^d \varepsilon^{s,e}_i\right) \left (\frac{1}{b-s} \sum_{i=s+1}^b f_i-\frac{1}{e-{\eta_k}} \sum_{i={\eta_{k}+1 } }^e f_i\right)  \right|\\
		\le&\sqrt {b-{\eta_k}}    \gamma \left( \left| f_{\eta_k} -f_{\eta_{k+1} }\right| +  \left|\frac{1}{b-s} \sum_{i=s+1 }^b f_i- f_{\eta_k}  \right| + \left |\frac{1}{e-{\eta_k}} \sum_{i={\eta_k+1}}^e f_i - f_{\eta_{k+1} } \right| \right)\\
		\le & \sqrt {b-{\eta_k}} (  \kse + (C_R +1)\kse +(C_R+1)\kse),
	\end{align*}
	where the first inequality follows from \eqref{eq:wbs size of intervals} and \eqref{eq:wbs noise 2}, and the second inequality from \Cref{lem-cov-16}.

\vskip 3mm
\noindent {\bf Step 3.3.}  Finally, we have that
	\begin{align*}
		III = \sqrt{e-b}\left(\frac{1}{e-b}\sum_{i = b+1}^e \varepsilon^{s, e}_i\right)\left(\frac{1}{e-\eta_k}\sum_{i=\eta_k + 1}^e f_i - \frac{1}{e-b}\sum_{i=b+1}^e f_i\right).	
	\end{align*}
	Therefore,
	\begin{align*}
		|III| \leq \sqrt{e-b}\gamma \frac{b-\eta_k}{e-b}(C_R+1)\kse	 \leq \frac{4\sqrt{2}}{\min\{1, c_1\}}|b-\eta_k | \Delta^{-1/2} \gamma  (C_R+1) \kse.
	\end{align*}

\vskip 3mm

\noindent {\bf Step 4.} Using the first part of \Cref{lem-cov-11}, the first term on the right hand side of \eqref{eq:perturbations} can be bounded as
	\[
		\langle \varepsilon^{s,e}, \p_d(X^{s,e})-\p_d(f^{s,e})\rangle\le \gamma^2.
	\]
	Thus \eqref{eq:WBS sufficient 2} holds if 
	\[
		|b-\eta_k|(\kse)^2 \ge  C  \max\left\{  |b-\eta_k| \Delta^{-1/2} \gamma \kse  ,\quad  \sqrt {b-{\eta_k}}    \gamma \kse  ,\quad \gamma^2 \right\}.
	\]
	Since $\gamma \le c_3\sqrt \Delta \kappa  $, the first inequality holds. The second inequality follows from $|b-\eta_k| \ge C_3\gamma^2 (\kappa_k)^{-2} \ge C_3\gamma^2 (\kse)^{-2}$, as assumed in \eqref{eq:wbs contradict assume}.  This completes the proof.
\end{proof}

\bibliographystyle{ims}
\bibliography{citations}

\begin{thebibliography}{46}
\expandafter\ifx\csname natexlab\endcsname\relax\def\natexlab#1{#1}\fi
\expandafter\ifx\csname url\endcsname\relax
  \def\url#1{\texttt{#1}}\fi
\expandafter\ifx\csname urlprefix\endcsname\relax\def\urlprefix{URL }\fi
\providecommand{\eprint}[2][]{\url{#2}}

\bibitem[{Aston and Kirch(2014)}]{AstonKirch2014}
\textsc{Aston, J. A.~D.} and \textsc{Kirch, C.} (2014).
\newblock Efficiency of change point tests in high dimensional settings.
\newblock \textit{arXiv preprint arXiv: 1409.1771}.

\bibitem[{Aue et~al.(2009)Aue, H\"{o}mann, Horv\'{a}th and
  Reimherr}]{AueEtal2009}
\textsc{Aue, A.}, \textsc{H\"{o}mann, S.}, \textsc{Horv\'{a}th, L.} and
  \textsc{Reimherr, M.} (2009).
\newblock Break detection in the covariance structure of multivariate nonlinear
  time series models.
\newblock \textit{The Annals of Statistics}, \textbf{37} 4046--4087.

\bibitem[{Avanesov and Buzun(2016)}]{avanesov2016change}
\textsc{Avanesov, V.} and \textsc{Buzun, N.} (2016).
\newblock Change-point detection in high-dimensional covariance structure.
\newblock \textit{arXiv preprint arXiv:1610.03783}.

\bibitem[{Baranowski et~al.(2016)Baranowski, Chen and
  Fryzlewicz}]{BaranowskiEtal2016}
\textsc{Baranowski, R.}, \textsc{Chen, Y.} and \textsc{Fryzlewicz, P.} (2016).
\newblock Narrowest-{O}ver-{T}hreshold detection of multiple change-points and
  change-point-like feature.
\newblock \textit{arXiv preprint arXiv: 1609.00293}.

\bibitem[{Boysen et~al.(2009)Boysen, Kempe, Liebscher, Munk and
  Wittich}]{BoysenEtal2009}
\textsc{Boysen, L.}, \textsc{Kempe, A.}, \textsc{Liebscher, V.}, \textsc{Munk,
  A.} and \textsc{Wittich, O.} (2009).
\newblock Consistencies and rates of convergence of jump-penalized least
  squares estimators.
\newblock \textit{The Annals of Statistics}, \textbf{37} 157--183.

\bibitem[{Chan and Walther(2013)}]{chan2013}
\textsc{Chan, H.~P.} and \textsc{Walther, G.} (2013).
\newblock Detection with the scan and the average likelihood ratio.
\newblock \textit{Statistica Sinica}, \textbf{1} 409--428.

\bibitem[{Cho(2015)}]{Cho2015}
\textsc{Cho, H.} (2015).
\newblock Change-point detection in panel data via double cusum statistic.
\newblock \textit{Electronic Journal of Statistics} in press.

\bibitem[{Cho and Fryzlewicz(2015)}]{ChoFryzlewicz2015}
\textsc{Cho, H.} and \textsc{Fryzlewicz, P.} (2015).
\newblock Multiple change-point detection for high-dimensional time series via
  {S}parsified {B}inary {S}egmentation.
\newblock \textit{Journal of the Royal Statistical Society: Series B
  (Statistical Methodology)}, \textbf{77} 475--507.

\bibitem[{Davis et~al.(2006)Davis, Lee and Rodriguez-Yam}]{DavisEtal2006}
\textsc{Davis, R.~A.}, \textsc{Lee, T. C.~M.} and \textsc{Rodriguez-Yam, G.~A.}
  (2006).
\newblock Structural break estimation for nonstationary time series models.
\newblock \textit{Journal of the American Statistical Association},
  \textbf{101} 223--239.

\bibitem[{D\"{u}mbgen and Spokoiny(2001)}]{dumbgen2001multiscale}
\textsc{D\"{u}mbgen, L.} and \textsc{Spokoiny, V.~G.} (2001).
\newblock Multiscale testing of qualitative hypotheses.
\newblock \textit{Annals of Statistics} 124--152.

\bibitem[{D{\"u}mbgen and Walther(2008)}]{dumbgen2008multiscale}
\textsc{D{\"u}mbgen, L.} and \textsc{Walther, G.} (2008).
\newblock Multiscale inference about a density.
\newblock \textit{The Annals of Statistics}, \textbf{36} 1758--1785.

\bibitem[{Eichinger and Kirch(2018)}]{EichingerKirch2018}
\textsc{Eichinger, B.} and \textsc{Kirch, C.} (2018).
\newblock A mosum procedure for the estimation of multiple random change
  points.
\newblock \textit{Bernoulli}, \textbf{24} 526--564.

\bibitem[{Enikeeva et~al.(2018)Enikeeva, Munk and Werner}]{enikeeva2018bump}
\textsc{Enikeeva, F.}, \textsc{Munk, A.} and \textsc{Werner, F.} (2018).
\newblock Bump detection in heterogeneous gaussian regression.
\newblock \textit{Bernoulli}, \textbf{24} 1266--1306.

\bibitem[{Fan and Guan(2017)}]{FanGuan2017}
\textsc{Fan, Z.} and \textsc{Guan, L.} (2017).
\newblock Approximate $ l_0 $-penalized estimation of piecewise-constant
  signals on graphs.
\newblock \textit{arXiv preprint}.

\bibitem[{Frick et~al.(2014)Frick, Munk and Sieling}]{FrickEtal2014}
\textsc{Frick, K.}, \textsc{Munk, A.} and \textsc{Sieling, H.} (2014).
\newblock Multiscale change point inference.
\newblock \textit{Journal of the Royal Statistical Society: Series B
  (Statistical Methodology)}, \textbf{76} 495--580.

\bibitem[{Friedrich et~al.(2008)Friedrich, Kempe, Liebscher and
  Winkler}]{FriedrichEtal2008}
\textsc{Friedrich, F.}, \textsc{Kempe, A.}, \textsc{Liebscher, V.} and
  \textsc{Winkler, G.} (2008).
\newblock Complexity penalized m-estimation: Fast computation.
\newblock \textit{Journal of Computational and Graphical Statistics},
  \textbf{17} 201--204.

\bibitem[{Fryzlewicz(2014)}]{fryzlewicz2014wild}
\textsc{Fryzlewicz, P.} (2014).
\newblock Wild binary segmentation for multiple change-point detection.
\newblock \textit{The Annals of Statistics}, \textbf{42} 2243--2281.

\bibitem[{Gao et~al.(2017)Gao, Han and Zhang}]{GaoEtal2017}
\textsc{Gao, C.}, \textsc{Han, F.} and \textsc{Zhang, C.~H.} (2017).
\newblock Minimax risk bounds for piecewise constant models.
\newblock \textit{arXiv preprint arXiv:1705.06386.}

\bibitem[{Gibberd and Roy(2017)}]{GibberdRoy2017}
\textsc{Gibberd, A.~J.} and \textsc{Roy, S.} (2017).
\newblock Multiple changepoint estimation in high-dimensional {G}aussian
  graphical models.
\newblock \textit{arXiv preprint}.

\bibitem[{James et~al.(1987)James, James and Siegmund}]{JamesEtal1987}
\textsc{James, B.}, \textsc{James, K.~L.} and \textsc{Siegmund, D.} (1987).
\newblock Tests for a change-point.
\newblock \textit{Biometrika}, \textbf{74} 71--83.

\bibitem[{Jeng et~al.(2012)Jeng, Cai and Li}]{jeng2012simultaneous}
\textsc{Jeng, X.~J.}, \textsc{Cai, T.~T.} and \textsc{Li, H.} (2012).
\newblock Simultaneous discovery of rare and common segment variants.
\newblock \textit{Biometrika}, \textbf{100} 157--172.

\bibitem[{Jirak(2015)}]{Jirak2015}
\textsc{Jirak, M.} (2015).
\newblock Uniform change point tests in high dimension.
\newblock \textit{The Annals of Statistics}, \textbf{43} 2451--2483.

\bibitem[{Killick et~al.(2012)Killick, Fearnhead and Eckley}]{KillickEtal2012}
\textsc{Killick, R.}, \textsc{Fearnhead, P.} and \textsc{Eckley, I.~A.} (2012).
\newblock Optimal detection of changepoints with a linear computational cost.
\newblock \textit{Journal of the American Statistical Association},
  \textbf{107} 1590--1598.

\bibitem[{Lavielle(1999)}]{Lavielle1999}
\textsc{Lavielle, M.} (1999).
\newblock Detection of multiple changes in a sequence of dependent variables.
\newblock \textit{Stochastic Processes and their Applications}, \textbf{83}
  79--102.

\bibitem[{Lavielle and Moulines(2000)}]{LavielleMoulines2000}
\textsc{Lavielle, M.} and \textsc{Moulines, E.} (2000).
\newblock Least-squares estimation of an unknown number of shifts in a time
  series.
\newblock \textit{Journal of Time Series Analysis}, \textbf{21} 33--59.

\bibitem[{Li et~al.(2017)Li, Guo and Munk}]{LiEtal2017}
\textsc{Li, H.}, \textsc{Guo, Q.} and \textsc{Munk, A.} (2017).
\newblock Multiscale change-point segmentation: Beyond step functions.
\newblock \textit{arXiv preprint arXiv: 1708.03942}.

\bibitem[{Liebscher and Winkler(1999)}]{LiebscherWinkler1999}
\textsc{Liebscher, V.} and \textsc{Winkler, G.} (1999).
\newblock A potts model for segmentation and jump-detection.
\newblock In \textit{Proceedings S4G International Conference on Stereology,
  Spatial Statistics and Stochastic Geometry, Prague}, vol.~21.

\bibitem[{Maidstone et~al.(2017)Maidstone, Hocking, Rigaill and
  Fearnhead}]{MaidstoneEtal2017}
\textsc{Maidstone, R.}, \textsc{Hocking, T.}, \textsc{Rigaill, G.} and
  \textsc{Fearnhead, P.} (2017).
\newblock On optimal multiple changepoint algorithms for large data.
\newblock \textit{Statistics and Computing}, \textbf{27} 519--533.

\bibitem[{Page(1954)}]{Page1954}
\textsc{Page, E.~S.} (1954).
\newblock Continuous inspection schemes.
\newblock \textit{Biometrika}, \textbf{41} 100--115.

\bibitem[{Rigaill(2010)}]{Rigaill2010}
\textsc{Rigaill, G.} (2010).
\newblock Pruned dynamic programming for optimal multiple change-point
  detection.
\newblock \textit{arXiv preprint arXiv:1004.0887}.

\bibitem[{Rinaldo(2009)}]{Rinaldo2009}
\textsc{Rinaldo, A.} (2009).
\newblock Properties and refinements of the fused lasso.
\newblock \textit{The Annals of Statistics}, \textbf{37} 2292--2952.

\bibitem[{Scott and Knott(1974)}]{ScottKnott1974}
\textsc{Scott, A.~J.} and \textsc{Knott, M.} (1974).
\newblock A cluster analysis method for grouping means in the analysis of
  variance.
\newblock \textit{Biometrics} 507--512.

\bibitem[{Tibshirani et~al.(2005)Tibshirani, Saunders, Rosset, Zhu and
  Knight}]{TibshiraniEtal2005}
\textsc{Tibshirani, R.}, \textsc{Saunders, M.}, \textsc{Rosset, S.},
  \textsc{Zhu, J.} and \textsc{Knight, K.} (2005).
\newblock Sparsity and smoothness via the fused lasso.
\newblock \textit{Journal of the Royal Statistical Society: Series B
  (Statistical Methodology)}, \textbf{67} 91--108.

\bibitem[{Tickle et~al.(2018)Tickle, Eckley, Fearnhead and
  Haynes}]{TickleEtal2018}
\textsc{Tickle, S.~O.}, \textsc{Eckley, I.~A.}, \textsc{Fearnhead, P.} and
  \textsc{Haynes, K.} (2018).
\newblock Parallelisation of a common changepoint detection method.
\newblock \textit{arXiv preprint arXiv:1810.03591}.

\bibitem[{Tsybakov(2009)}]{Tsybakov2009}
\textsc{Tsybakov, A.~B.} (2009).
\newblock \textit{Introduction to Nonparametric Estimation}.
\newblock Springer.

\bibitem[{Venkatraman(1992)}]{venkatraman1992consistency}
\textsc{Venkatraman, E.~S.} (1992).
\newblock \textit{Consistency results in multiple change-point problems}.
\newblock Ph.D. thesis, Stanford University.

\bibitem[{Vershynin(2010)}]{vershynin2010introduction}
\textsc{Vershynin, R.} (2010).
\newblock Introduction to the non-asymptotic analysis of random matrices.
\newblock \textit{arXiv preprint arXiv:1011.3027}.

\bibitem[{Vostrikova(1981)}]{vostrikova1981detection}
\textsc{Vostrikova, L.} (1981).
\newblock Detection of the disorder in multidimensional random-processes.
\newblock \textit{Doklady Akademii Nauk SSSR}, \textbf{259} 270--274.

\bibitem[{Wald(1945)}]{Wald1945}
\textsc{Wald, A.} (1945).
\newblock Sequential tests of statistical hypotheses.
\newblock \textit{The Annals of Mathematical Statistics}, \textbf{16} 117--186.

\bibitem[{Wang et~al.(2017)Wang, Yu and Rinaldo}]{WangEtal2017}
\textsc{Wang, D.}, \textsc{Yu, Y.} and \textsc{Rinaldo, A.} (2017).
\newblock Optimal covariance change point detection in high dimension.
\newblock \textit{arXiv preprint}.

\bibitem[{Wang et~al.(2018)Wang, Yu and Rinaldo}]{WangEtal2018}
\textsc{Wang, D.}, \textsc{Yu, Y.} and \textsc{Rinaldo, A.} (2018).
\newblock Optimal change point detection and localization in sparse dynamic
  networks.
\newblock \textit{arXiv preprint arXiv:1809.09602}.

\bibitem[{Wang and Samworth(2018)}]{wang2016high}
\textsc{Wang, T.} and \textsc{Samworth, R.~J.} (2018).
\newblock High-dimensional changepoint estimation via sparse projection.
\newblock \textit{Journal of the Royal Statistical Society: Series B
  (Statistical Methodology)}.

\bibitem[{Yao(1988)}]{Yao1988}
\textsc{Yao, Y.~C.} (1988).
\newblock Estimating the number of change-points via schwarz' criterion.
\newblock \textit{Statistics \& Probability Letters}, \textbf{6} 181--189.

\bibitem[{Yao and Au(1989)}]{YaoAu1989}
\textsc{Yao, Y.-C.} and \textsc{Au, S.-T.} (1989).
\newblock Least-squares estimation of a stop function.
\newblock \textit{Sankhy\={a}: The Indian Journal of Statistics, Series A}
  370--381.

\bibitem[{Yao and Davis(1986)}]{YaoDavis1986}
\textsc{Yao, Y.~C.} and \textsc{Davis, R.~A.} (1986).
\newblock The asymptotic behavior of the likelihood ratio statistic for testing
  a shift in mean in a sequence of independent normal variates.
\newblock \textit{Sankhy\={a}: The Indian Journal of Statistics, Series A}
  339--353.

\bibitem[{Yu(1997)}]{yu1997assouad}
\textsc{Yu, B.} (1997).
\newblock \textit{Festschrift for Lucien Le Cam}, vol. 423, chap. Assouad,
  {F}ano, and {L}e {C}am.
\newblock Springer Science \& Business Media, 435.

\end{thebibliography}

\end{document}